\documentclass[10pt,oneside, reqno]{amsart}
\usepackage[utf8]{inputenc}
\usepackage[T1]{fontenc} 
\usepackage{amsmath, amscd}
\usepackage{amsthm}
\usepackage{amsfonts}
\usepackage{bm}
\usepackage{amssymb}
\usepackage{geometry}
\usepackage{graphicx}
\usepackage{url, enumerate}
\usepackage{hyperref} 
\usepackage{tikz}
\usetikzlibrary{positioning}
\usepackage{subcaption}
\usepackage{nameref}
\usepackage[shortlabels]{enumitem}

\theoremstyle{plain}
\newtheorem{theorem}{Theorem}[section]
\newtheorem*{theorem*}{Theorem}
\newtheorem{lemma}[theorem]{Lemma}
\newtheorem{fact}[theorem]{Fact}
\newtheorem{proposition}[theorem]{Proposition}
\newtheorem{cor}[theorem]{Corollary}

\theoremstyle{definition}
\newtheorem{definition}[theorem]{Definition}
\newtheorem{remark}[theorem]{Remark}
\newtheorem{example}[theorem]{Example}
\newtheorem*{notation}{Notation}

\newtheorem*{property}{Property}

\def\forkindep{\mathrel{\raise0.2ex\hbox{\ooalign{\hidewidth$\vert$\hidewidth\cr\raise-0.9ex\hbox{$\smile$}}}}}
\def\Ind#1#2{#1\setbox0=\hbox{$#1x$}\kern\wd0\hbox to 0pt{\hss$#1\mid$\hss}
\lower.9\ht0\hbox to 0pt{\hss$#1\smile$\hss}\kern\wd0}
\def\Notind#1#2{#1\setbox0=\hbox{$#1x$}\kern\wd0\hbox to 0pt{\mathchardef
\nn="3236\hss$#1\nn$\kern1.4\wd0\hss}\hbox to
0pt{\hss$#1\mid$\hss}\lower.9\ht0
\hbox to 0pt{\hss$#1\smile$\hss}\kern\wd0}
\def\ind{\mathop{\mathpalette\Ind{}}}

\makeatletter
\def\namedlabel#1#2{\begingroup
    #2%
    \def\@currentlabel{#2}%
    \phantomsection\label{#1}\endgroup
}
\makeatother

\DeclareMathOperator{\tp}{tp}
\DeclareMathOperator{\qftp}{qftp}

\DeclareMathOperator{\dcl}{dcl}
\DeclareMathOperator{\acl}{acl}
\def\acleq{\text{acl}^{eq}}

\DeclareMathOperator{\aut}{Aut}
\DeclareMathOperator{\cb}{Cb}
\DeclareMathOperator{\hb}{Hb}
\DeclareMathOperator{\bdd}{bdd}
\def\sseq{\subseteq}
\def\eq{^{eq}}

\def\rest{\upharpoonright}

\def\G00{G^{00}}
\newcommand{\opart}[1]{\widetilde{#1}}

\def\frP{\mathfrak{P}}

\def\mcL{\mathcal{L}}
\def\mcC{\mathcal{C}}

\def\mcF{\mathcal{F}}
\def\mcG{\mathcal{G}}

\title{Transfer of generalized amalgamation in simple theories
}
\author{Baptiste Schilling}
\date{}

\begin{document}

\begin{abstract}
We give an abstract framework to transfer generalized amalgamation from a simple theory to another, and we apply it to theories of bounded PAC structures, of fields with operators and of lovely pairs.
We show in particular that bounded pseudo-algebraically closed fields have generalized amalgamation, regardless of their imperfection degree.
\end{abstract}

\maketitle

\let\thefootnote\relax\footnotetext{\textit{Date.} \today}

\section{Introduction}

In \cite{KimPillay1997}, Kim and Pillay characterized simple theories by the existence of a well-behaved independence relation, called nonforking independence. 
This relation satisfies several properties, one of which - the independence theorem - can be rewritten as an existence property on three types.
In this regard, the property of $n$-existence for $n$ greater than $3$ was introduced by Shelah in \cite{Shelah_83_ClassificationTheory}, and was later used by Kolesnikov in \cite{Kolesnikov_2004_thesis} to refine the classification of simple unstable theories.
Notable results were obtained in \cite{DePiro_Kim_Millar_2005}, in which the authors show that all stable theories with elimination of imaginaries have $n$-existence for all $n$ over elementary substructures, and link the $4$-existence property to the existence of a hyperdefinable group out of a group configuration diagram.

Many simple theories have been obtained by modifying some already well-studied stable theories, for example by expanding their language in a controlled way.
Among others, this is the case for the theory ACFA of algebraically closed fields endowed with a generic automorphism, for which Chatzidakis and Hrushowski showed in \cite{ChatzidakisHrushovski_1999} that it has $n$-existence for every integer $n$, and hence that it is simple.

Another well-studied theory of fields is that of pseudo-algebraically closed fields: these objects were introduced in \cite{Ax1973} by Ax as fields $K$ such that any absolutely irreducible variety defined over $K$ has a $K$-rational point.
Hrushovski showed in \cite{HrushovskiPFFieldsRelStructures} that perfect bounded pseudo-algebraically closed fields are simple, which was later generalized to all bounded pseudo-algebraically closed fields by Chatzidakis and Pillay in \cite{Chatzidakis_Pillay_1998}.
Conversely, Chatzidakis showed in \cite{Chatzidakis_1999} that unbounded pseudo-algebraically closed fields are not simple.

In his initial proof of simplicity in \cite{HrushovskiPFFieldsRelStructures}, Hrushovski defined a general notion of pseudo-algebraically closed (PAC) substructures in a given strongly minimal theory.
In \cite{Pillay_Polkowska_2006}, Pillay and Polkoswska extended Hrushovski's notion of pseudo-algebraically closed substructure to stable theories. 
Polkowska then showed in \cite{Polkowska_2007} that bounded PAC substructures are simple by proving that they satisfy the independence theorem, and hence by characterizing the nonforking independence in bounded PAC substructures.

\smallskip

In this note, we generalize this result of Polkowska to show that bounded PAC substructures of a stable theory have $n$-existence for all $n$ greater than $3$. We do so by presenting a framework allowing the transfer of $n$-existence from a simple theory $T_0$ to another theory $T_1$.

We start in section \ref{sec:GeneralizedAmalgamation} by defining generalized amalgamation systems and the notion of $n$-existence, as well as stating some of the main results of the literature. 
In section \ref{sec:TransferNExistence} we introduce an abstract setting in which it is possible to transfer generalized existence from a simple theory $T_0$ to another simple theory $T_1$  in a expanded language such that every model of $T_1$ embeds into a model of $T_0$. 
This setting takes the form of a series of strong hypotheses, under which we are able to prove the following theorem:

\medskip

\paragraph*{\textbf{Theorem }(Theorem \ref{thm:AbstractAmalgamationTransfer})}
\textit{
Let $T_0$ and $T_1$ be two simple theories such that $T_0$ eliminates quantifiers and  $T_0^\forall$ is contained in $T_1$. 
Fix a model $N$ of $T_1$ contained in a model $M$ of $T_0$, both sufficiently saturated.
Assume that the hypotheses of subsection \ref{sec:MainThm} hold (i.e. the hypotheses \ref{hyp:StrongBdd} to \ref{hyp:ContinuityFSystems}) and let $Q$ be a subset of $N$ algebraically closed in the theory $T_1$.
If $T_0$ has $n$-existence over the algebraic closure $\acl_0(Q)$ of $Q$ in $T_0$, then $T_1$ has $n$-existence over $Q$.
}
\medskip

The hypotheses \ref{hyp:StrongBdd} to \ref{hyp:ContinuityFSystems} are chosen so that the proof of the theorem works and are not necessarily intuitive, which is why we then discuss some ways in which they arise more naturally, leaving \ref{hyp:ContinuityFSystems} as the only one whose verification presents difficulties.
The rest of this note is then dedicated to giving examples of applications in each of these cases.

\smallskip

In Section \ref{sec:BddPAC}, we apply the above theorem to the case where $T_1$ is the theory of a bounded PAC substructure of a structure whose theory $T_0$ is stable.
Our context is very close to the one developped in  \cite{Polkowska_2007}, where Polkowska showed that the theory $T_1$ is simple; we generalize this result to obtain that $T_1$ has $n$-existence for all $n$.
In our context however, elimination of imaginaries is replaced by two weaker assumptions: a technical one called (\nameref{Star}) (see Subsection \ref{sec:RelUniq}) as well as the fact that types over algebraically closed sets are stationary, which has been for example considered by Bartnick in \cite{bartnick_2024}, among others.

\medskip

\paragraph{\textbf{Corollary} (Corollary \ref{cor:nExistencePAC})}
\textit{
Let $T_0$ be a stable theory with quantifier elimination, in which the PAC-property is first-order (see Definition \ref{def:PACFO}) and in which types over algebraically closed sets are stationary.
Let $T_1$ be the theory of a bounded PAC substructure of a model of $T_0$. 
Assume that Property (\nameref{Star}) holds.
If $Q$ is an existentially closed subset of a model of $T_1$ containing an elementary substructure such that $T_0$ has $n$-existence over $\acl_0(Q)$, then $T_1$ has $n$-existence over $Q$.
}

\medskip

We then consider bounded PAC substructures in concrete theories of stable fields, namely that of differentially and separably closed fields, to obtain the following result:

\medskip

\paragraph{\textbf{Corollary} (Corollaries \ref{SCFfinite}, \ref{cor:DCF0PAC} and \ref{cor:namalgSCFinfiniPAC})}
\textit{
All bounded PAC substructures of the following theories have $n$-existence over elementary substructures for all $n$:
\begin{enumerate}
	\item The theory $DCF_{0,m}$ of differentially closed fields of characteristic $0$ with $m$ commuting derivations, in the language of rings with $m$ derivations.
	\item The theory $SCF_{p,e}$ ($e$ finite) of separably closed fields of characteristic $p$ and imperfection degree $e$, in the language of rings with $\lambda$-functions relative to a fixed $p$-basis.
	\item The theory $SCF_{p, \infty}$ of separably closed fields of characteristic $p$ and infinite imperfection degree, in the language of rings with generalized $\lambda$-functions.
\end{enumerate}
}

\medskip

In particular, the characterization of PAC substructures of $SCF_{p,e}$ and $SCF_{p, \infty}$ obtained by Hoffmann and Kowalski in \cite{Hoffmann_Kowalski_2023} allows us to conclude that all pseudo-algebraically closed fields have $n$-existence over elementary substructures for all $n$.

In Section \ref{sec:FieldsOperators}, we consider the case of an expansion of $T_0$ by unary function symbols, with a focus on theories of fields with operators.
Our context allows us to generalize the proof of $n$-existence in the theory ACFA in \cite{ChatzidakisHrushovski_1999} in order to obtain $n$-existence in several expansions of the theories of algebraically and separably closed fields by unary function symbols.

Finally in Section \ref{sec:PairsOfStructures}, we propose a context containing theories several theories of pairs (namely $H$-structures on supersimple theories \cite{BerensteinCarmonaVassiliev_2017}, lovely pairs of simple theories \cite{BenyaacovPillayVassiliev2003,Poizat_1983} and $\kappa$-PAC beautiful pairs in a stable theory \cite{Pillay_Polkowska_2006, Polkowska_2007}), and we show that $n$-existence in the initial theory transfers to such a theory of pairs.

\bigskip
Throughout this note, we assume a certain familiarity with the main tools of geometric model theory.

\subsection*{Acknowledgments}
I would like to thank my supervisors, Thomas Blossier and Amador Martin-Pizarro, for their help and advice during the redaction of this note.

\smallskip

\section{Generalized amalgamation}\label{sec:GeneralizedAmalgamation}

\subsection{Definitions}		
\label{sec:DefAmalgamation}

We consider a sufficiently saturated (infinite) model $M$ of a complete theory $T$  in a language $\mcL$. As usual, all subsets and tuples are small with respect to the saturation of $M$, unless stated. 
In all of this note, a tuple need not be finite and will be denoted by lowercase letters.

Kim-Pillay's theorem \cite{KimPillay1997} states that the theory $T$ is simple if and only if  there is a ternary relation $\ind$ between subsets (or rather tuples) of $M$ satisfying the following properties: 

\begin{itemize}[leftmargin = 0 cm]
	\item[]\namedlabel{ppt:invar}{INVARIANCE} $~$ If $(a,b,c) \equiv (a',b',c')$,  then $a\ind_c b$ holds if and only if $a'\ind_{c'} b'$.
	
	\item[]\namedlabel{ppt:sym}{SYMMETRY}  $~$ If $a\ind_c b$, then $b\ind_c a$.
	
	\item[]\namedlabel{ppt:finchar}{FINITE CHARACTER}  $~$ $a\ind_c b$ if and only if $a_0 \ind _c b$ for all finite $a_0$ contained in $a$.
	
	\item[]\namedlabel{ppt:locchar}{LOCAL CHARACTER}  $~$ For any $a$ and $b$ there exist $c$ contained in $b$ such that $|c|\leq |T|$ and $a\ind_c b$.
	
	\item[]\namedlabel{ppt:transit}{TRANSITIVITY}  $~$ If $a\ind_c b$ and $a\ind_{bc} d$, then $a\ind_c bd$.
	
	\item[]\namedlabel{ppt:monot}{MONOTONICITY}  $~$ If $a\ind_c bd$, then $a\ind_c b$.
	
	\item[]\namedlabel{ppt:basemonot}{BASE MONOTONICITY}  $~$ If $a\ind_c bd$, then $a\ind_{cb} d$.
	
	\item[]\namedlabel{ppt:extension}{EXTENSION}  $~$ For all $a,b$ and $c$ there exist an $a'$ with $a'\equiv_c a$ such that $a'\ind_c b$.

	\item[]\namedlabel{ppt:indepthm}{INDEPENDENCE THEOREM} (over any elementary substructure $c$) $~$ If $a_1\equiv_c a_2$, $a_1\ind_c b_1$, $a_2\ind_c b_2$ and $b_1\ind_c b_2$ hold, then there exist $a$ such that $a\equiv_{b_1}a_1$, $a\equiv_{b_2} a_2$ and $a\ind_c b_1b_2$.
\end{itemize}
In this case, the relation $\ind$ coincides with the nonforking independence of the theory $T$ and satisfies 
\begin{itemize}[leftmargin = 0 cm]
    	\item[]\namedlabel{ppt:closure}{CLOSURE}  $~$ $a\ind_c b$ if and only if $\acl(ac)\ind_{\acl(c)} \acl(bc)$.
    
\end{itemize}

Moreover, the simple theory is stable if and only if nonforking independence $\ind$ also satisfies the following: 

\begin{itemize}[leftmargin = 0 cm]
    \item[]\namedlabel{ppt:stationarity}{STATIONARITY} (over any elementary substructure $c$)  $~$ If $a \equiv_c a'$, $a\ind_c b$ and $a'\ind_c b$ hold, then $a\equiv_b a'$.
\end{itemize}

\begin{remark}\label{rk:3AmalgEquivIndepThm}
Given a simple theory $T$, the independence theorem over any algebraically closed subset $C$ can be reformulated as follows (which is usually referred to as \emph{3-existence}): 
For all types  $p_1(x_1)$, $p_2(x_2)$, $p_3(x_3)$, $p_{1,2}(x_{1,2})$, $p_{1,3}(x_{1,3})$ and $p_{2,3}(x_{2.3})$ over $C$ satisfying the following conditions:
\begin{itemize}
    \item For all $i<j$, both tuples of variables $x_i$ and $x_j$ are contained in $x_{i,j}$.
    \item For all $i<j$, the realizations of $p_i$ and of $p_{i,j}$ are algebraically closed.
    \item $p_{i,j}(x_{i,j}) \models p_i(x_i)\cup p_j(x_j) \cup "x_i \ind_C x_j"\cup "x_{i,j}=\acl(x_ix_j)"$.
\end{itemize}
There exists a type $p_{1,2,3}(x_{1,2,3})$ extending all $p_{i,j}$'s such that whenever $a_{1,2,3}$ realizes $p_{1,2,3}$, setting $a_{i}=a_{1,2,3}\rest x_i$, we have that the family $a_1, a_2, a_3$ is independent over $C$.

Indeed, it suffices to set $p_{1,2}=\tp(\acl(C,b_1, b_2)/C)$, $p_{1,3} = \tp(\acl(C,b_1, a_1)/C)$ and $p_{2,3}= \tp(\acl(C, b_2, a_2)/C)$ and $p_{1,2,3} = \tp(\acl(C,a, b_1, b_2)/C)$ to see the equivalence, where $p_1=\tp(\acl(C, b_1)/C)$, $p_2=\tp(\acl(C, b_2)/C)$ and $p_3= \tp(\acl(C,a_1)/C)$ (which is also equal to $\tp(\acl(C,a_2)/C)$).
\end{remark}

We can now describe the $n$-existence property, which generalizes  the $3$-existence property for $n\geq 3$. For this, we need some auxiliary notions for our simple theory $T$.

\begin{notation}
For any positive integer $n\ge 1$, denote by $[n]$ the interval $\{1,...,n\}$ and by $\mathfrak{P}([n])^-$ the set $\mathfrak{P}([n]) \setminus \{[n]\}$.

\end{notation}

\begin{definition}\label{AmalgSyst}
Let $W$ be a finite subset of $\mathfrak{P}(\mathbb{N})$ which is closed under subsets. 
A \newline  \emph{$W$-amalgamation system} over a set $C$ of parameters is a family $S_W$ of types over  $C$ of the form  $S_W=(p_w(x_w), w\in W)$  verifying the following conditions:
\begin{enumerate}
	\item \label{axiom:AmalgSyst1} For any two distinct variables $y, z$ appearing in $x_w$, we have $p_w\vdash y\neq z$.
	
	\item \label{axiom:AmalgSyst2} For all $v$ and $w$ in $W$ we have $x_v\cap x_w = x_{v\cap w}$. If $v$ is contained in $w$, then $p_w(x_w) \rest x_v = p_v(x_v)$.
	
	\item Given an element $w$ of $W$ and a realization $b_w$ of $p_w(x_w)$, set  $b_i=b_w \rest x_{\{i\}}$ for $i$ in $w$.  The following holds:
	\begin{enumerate}
		\item \label{axiom:AmalgSystIndependence} \emph{Independence}: The family $(b_i)_{i\in w}$ is independent over $C$.
		
		\item \label{axiom:AmalgSystControl} \emph{Controlled character}: The tuple $b_w$ is algebraic over $C \cup \{b_i, i\in w\}$.
	\end{enumerate}
\end{enumerate}

We say that $S$ is an $n$-amalgamation system if $W$ is contained in $\mathfrak{P}([n])$.
The system $S$ is \emph{complete} if $W$ has a maximum $\top$ (with respect to inclusion $\sseq$). It can be \emph{completed} if there exists a type $p_\top$ such that $S \cup \{p_\top\}$ is a complete amalgamation system, where $\top$ is the maximum of $W\cup \{\top\}$. 
In this case, the type $p_\top$ is a \emph{completion} of the system.

The system $S$ is \emph{algebraically closed} if $b_w = \acl(C\cup\{b_i, i\in w\})$ for every realization $b_w$ of $p_w$, with  $w$ in $W$. 
In particular, if $S$ is algebraically closed, then $C$ is algebraically closed and is contained in any realization of $p_w$ for all $w\in W$.
\end{definition}

Notice that the existence of a completion to an algebraically closed amalgamation system implies the existence of an algebraically closed completion: simply take the type of the algebraic closure of any realization of the completion.
Moreover, if $p_\top(x_\top)$ is a completion of the (not necessarily algebraically closed) amalgamation system $S=(p_w(x_w), w\in W)$, then the union $\bigcup_{w\in W} x_w$ of all the variables in $S$ is well defined and contained in $x_\top$.
It follows that the type $p_\top(x_\top)\rest \bigcup_{w\in W} x_w$ is also a completion of $S$.
In particular, it is enough to look for completions $p_\top(x_\top)$ with $x_\top =\bigcup_{w\in W} x_w$: we say that such a completion is \textit{minimal}.

\bigskip

\begin{center}
\begin{tikzpicture}

\node (A0) at (2.4,-1.3) {$p_\emptyset$};

\node (A1) {$p_1$};
\node (A2) [right= of A1] {$p_2$};
\node (A3) [right= of A2] {$p_3$};
\node (A4) [right= of A3] {$p_4$};

\node (A12) at (-2,1.5) {$p_{12}$};
\node (A13) [right= of A12] {$p_{13}$};
\node (A23) [right= of A13] {$p_{23}$};
\node (A14) [right= of A23] {$p_{14}$};
\node (A24) [right= of A14] {$p_{24}$};
\node (A34) [right= of A24] {$p_{34}$};

\node (A123) at (-0.5,3) {$p_{123}$};
\node (A124) [right= of A123] {$p_{124}$};
\node (A134) [right= of A124] {$p_{134}$};
\node (A234) [right= of A134] {$p_{234}$};

\node (A1234) at (2.4,4.5) {?};

\draw [->] (A0) -- (A1) node  {};
\draw [->] (A0) -- (A2) node  {};
\draw [->] (A0) -- (A3) node  {};
\draw [->] (A0) -- (A4) node  {};

\draw[->] (A1) -- (A12) node {};
\draw[->] (A1) -- (A13) node {};
\draw[->] (A1) -- (A14) node {};
\draw[->] (A2) -- (A12) node {};
\draw[->] (A2) -- (A23) node {};
\draw[->] (A2) -- (A24) node {};
\draw[->] (A3) -- (A13) node {};
\draw[->] (A3) -- (A23) node {};
\draw[->] (A3) -- (A34) node {};
\draw[->] (A4) -- (A14) node {};
\draw[->] (A4) -- (A24) node {};
\draw[->] (A4) -- (A34) node {};

\draw[->] (A12) -- (A123) node {};
\draw[->] (A13) -- (A123) node {};
\draw[->] (A23) -- (A123) node {};
\draw[->] (A12) -- (A124) node {};
\draw[->] (A14) -- (A124) node {};
\draw[->] (A24) -- (A124) node {};
\draw[->] (A13) -- (A134) node {};
\draw[->] (A14) -- (A134) node {};
\draw[->] (A34) -- (A134) node {};
\draw[->] (A23) -- (A234) node {};
\draw[->] (A24) -- (A234) node {};
\draw[->] (A34) -- (A234) node {};

\draw[->, dashed] (A123) -- (A1234) node  {};
\draw[->, dashed] (A124) -- (A1234) node  {};
\draw[->, dashed] (A134) -- (A1234) node  {};
\draw[->, dashed] (A234) -- (A1234) node  {};

\end{tikzpicture}
\end{center}
\nobreak
\begin{center}
A $4$-amalgamation system
\end{center}

\bigskip

\begin{definition}\label{nExist} \label{nUniq} 
We say that the theory $T$ has $W$-\textit{existence} (resp. $W$-\emph{uniqueness}) over a set of parameters $C$ if every algebraically closed $W$-amalgamation system over $C$ can be completed (resp. has at most one minimal completion).
Moreover, we say that $T$ has $n$-\textit{existence} (resp. $n$-\emph{uniqueness}) over $C$ if it has $W$-amalgamation (resp. $W$-uniqueness) over $C$ for all subsets $W$ of $\frP([n])$ that are closed under subsets.
\end{definition}

\begin{remark}
We chose to use the term "$n$-existence" instead of the more commonly used "$n$-amalgamation" in order to highlight the difference between $n$-existence and $n$-uniqueness.
\end{remark}

\noindent

It is immediate that $2$-existence is equivalent to the extension property. 
Remark \ref{rk:3AmalgEquivIndepThm} above shows that $3$-existence is equivalent to the independence theorem. 
In particular, every simple theory has $3$-existence over elementary substructures.
We also see that $2$-uniqueness over $C$ is equivalent to the stationarity of types over $C$; it follows that all stable theories with elimination of imaginaries have $2$-uniqueness over algebraically closed sets.

Note also that if $n$ equals $2$ or $3$, then $n$-existence implies that every $n$-amalgamation system has a completion (and not only the ones that are algebraically closed); this implication does not hold for $n \geq 4$ (for a counterexample, see \cite[example 1.1]{DePiro_Kim_Millar_2005}).

\begin{remark}
The $n$-existence property is equivalent to having $\mathfrak{P}([k])^-$-existence for all $k\leq n$.
In particular when trying to show $n$-existence for all $n$, it is sufficient to only consider $\mathfrak{P}([n])^-$-systems.
\end{remark}

\medskip

There are two differences between the notions of amalgamation presented here and others found in the literature.
The first is that we do not include the property of being algebraically closed in the definition of an amalgamation system: this will
allow us to consider amalgamation systems which are not algebraically closed without introducing a new definition.
The second difference is that we work with the real algebraic closure $\acl$ instead of the bounded closure $\bdd$ (as in \cite[Definition 1.4]{DePiro_Kim_Millar_2005}), or the imaginary algebraic closure $\acl\eq$ if the theory eliminates hyperimaginaries.
We always use the term "$n$-existence" to refer to the amalgamation of Definition \ref{nExist}; when confusion may arise, we use  "bounded existence" and "imaginary existence" respectively to refer to the other cases.

Note that Hrushovski in \cite{HrushovskiPFFieldsRelStructures} and Chatzidakis and Hrushovski in \cite{ChatzidakisHrushovski_1999} first introduce generalized amalgamation for the real algebraic closure before showing that the theories they are working with eliminate imaginaries. 
Our motivation to use real amalgamation here instead of imaginary amalgamation is that we want to work with theories for which no description of the imaginaries is known, such as the theory of separably closed fields of infinite imperfection degree (see Subsection \ref{sec:SCFPAC}).

Note also that if $\ind$ is a relation in a theory $T$ satisfying Invariance, Symmetry, Finite character, Local character, Transitivity, Monotonicity, Base monotonicity, Extension and Closure, then $3$-existence over elementary substructures is enough to prove that $T$ is simple and that $\ind$ is the nonforking independence by Kim-Pillay's theorem and Remark \ref{rk:3AmalgEquivIndepThm}.

\medskip

By definition, a theory $T$ has imaginary existence if and only if $T\eq$ has existence.
In particular, existence coincide with imaginary existence for theories with elimination of imaginaries.
In \cite{DePiro_Kim_Millar_2005}, the authors showed that all stable theories with elimination of imaginaries have existence over elementary substructures.
We will see in Subsection \ref{sec:RelUniq} that this result is preserved when weakening the assumption of elimination of imaginaries by simply assuming that types over algebraically closed sets are stationary.

\begin{fact} \cite[Proposition 1.6]{DePiro_Kim_Millar_2005} \label{Sharp}
Assume that $T$ is stable with elimination of imaginaries. 
Assume that the set of parameters $C$ satisfies the property $(\sharp_n)$:  the independence  $a\ind_C b_1...b_n$ implies 
	\[
	\dcl(\acl(ab_1C)...\acl(ab_nC))\cap \acl(b_1...b_nC) = \dcl(\acl(b_1C)...\acl(b_nC)).
	\] 
	In that case, the theory $T$ has $n$-existence over $C$.
	
\noindent
Moreover, the property $(\sharp_n)$ holds over all elementary substructures, hence the theory $T$ has $n$-existence over elementary substructures for all integers $n$.
\end{fact}

As of yet, we do not have any example of a stable theory (necessarily without elimination of imaginaries) which does not have $n$-existence over models.
In particular, we do not know if imaginary existence (over, say, models) implies existence, and vice-versa.

\subsection{First properties}
\label{sec:LemmaAmalgamation}

In this subsection we work in a complete simple theory $T$  and we gather some results on amalgamation systems.
These results are neither difficult nor surprising and were probably already known by the community; we include them with proofs here for the sake of completeness.

\begin{lemma} \label{ExtensionAmalgSystem}
Let $(p_w(x_w), w\in W)$ be an amalgamation system over a set of parameters $A$. If $B$ is a superset of $A$ such every $p_w(x_w)$ has a unique nonforking extension $q_w(x_w)$ to $B$, then $(q_w(x_w), w\in W)$ is an amalgamation system over $B$. 
\end{lemma}

\begin{proof}
To see that $(q_w(x_w), w\in W)$ is an amalgamation system, we first have to see that $q_w(x_w)\rest x_v = q_v(x_v)$ for $v\sseq w$.
But this follows from the uniqueness of $q_v$, as $q_w(x_w)\rest x_v$ is also a nonforking extension of $p_v$.

Fix now a realization $b_w$ of $q_w$. 
We know that $b_w$ is contained in the algebraic closure of the $b_i$'s together with $A$, and therefore in $\acl(Bb_i, i\in w)$.
Moreover, the family $(b_i, i\in w)$ is independent over $A$.
Since the independence $(b_i, i\in w) \ind_A B$ holds by hypothesis on $q_w$, this implies that $(b_i, i\in w)$ is an independent family over $B$.
\end{proof}

\begin{remark}
The hypothesis of this lemma is clearly fullfilled either when all $p_w$'s are stationary, or when the set $B$ is contained in the definable closure of $A$.
\end{remark}

An interesting consequence of this lemma is that if $C$ is a real algebraically closed set such that $\acl\eq(C)=\dcl\eq(C)$, then imaginary $3$-existence over $\acl\eq(C)$ implies $3$-existence over $C$.
To see this, remark that the unique extension of a $3$-amalgamation system over $C$ in $T$ to $\acl\eq(C)$ is a $3$-amalgamation system in $T\eq$, albeit not necessarily an algebraically closed one.
This does not matter here, since the $3$-existence property yields a completion to all $3$-existence systems as discussed in remark \ref{rk:3AmalgEquivIndepThm}
In particular we can find a completion of the extended system, which in turns yields a completion of the initial system.

\begin{lemma} \label{lem:RestrictionAmalgSystem}
Let $(p_w(x_w), w\in W)$ be an amalgamation system over a set of parameters $A$. If $C$ is a subset of $A$ such that for all $w$, 
the type $p_w$ does not fork over $C$, then $(p_w(x_w)\rest C, w\in W)$ is an amalgamation system over $C$. 
\end{lemma}

\begin{proof}
The compatibility of the $p_w(x_w)\rest C$'s is immediate, since the restriction $(p_w\rest C) \rest x_v$ is equal to $(p_w\rest x_v) \rest C = p_v \rest C$.

Fix a realization $b_w$ of $p_w(x_w)$ and notice that $b_w$ realizes $p_w\rest C$.
Since $(b_i, i\in w)$ is an independent family over $A$ (where $b_i:=b_w \rest x_i$) and the family $(b_i, i\in w)$ is independent from $A$ over $C$, we get that $(b_i, i\in w)$ is an independent family over $C$, which is witnessed by its type over $C$, and hence by $p_w\rest C$.

Remark that the independence $b_w \ind_{C} A$ yields $b_w \ind_{C, (b_i, i\in w)} A$.
Since $b_w$ is algebraic over $A\cup \{b_i, i\in w\}$, we get that $b_w$ is independent from itself over $C \cup \{b_i, i\in w\}$ and therefore algebraic over $C\cup \{b_i, i\in w\}$. 
This shows that $(p_w(x_w)\rest C, w\in W)$ is also controlled.
\end{proof}

\begin{lemma} \label{CompletionLemma}
Let $(q_w(x_w),w\in W)$ be an amalgamation system over an algebraically closed set of parameters $C$ such that for all $w$ in $W,$ 
\begin{itemize}
	\item if $a_w$ realizes $q_w$, then $\acl(Ca_w)=\dcl(Ca_w)$ and
	\item if $v$ is contained in $w$, then $\acl(Ca_v)\cap a_w = a_v$.
\end{itemize}
Then there exists an algebraically closed amalgamation system $(s_w(y_w), w\in W)$ over $C$ such that for all $w$ in $W$, we have the following:
\begin{itemize}
	\item $x_w \sseq y_w$ and $s_w(y_w) \rest x_w = q_w(x_w)$ and
	\item if $b_w$ realizes $s_w$, then $b_w= \acl(C\cup b_w \rest x_w)$.
\end{itemize}
In particular, every completion of $(s_w(y_w), w\in W)$ yields a completion of the initial system $(q_w(x_w),w\in W)$.
\end{lemma}

\begin{proof}
We define the $s_w$'s by induction on the size of $w$.
Let $w$ be an element of $W$, and assume that we have already defined the $s_v(y_v)$ with the desired properties for all $v$'s with $|v|<|w|$, and eventually for some $v$'s with $|v|=|w|$.
In particular, the $s_v(y_v)$'s, $v\subsetneq w$ have already been defined.

Fix a realization $a_w$ of $q_w$ and let $b_w$ be an arbitrary enumeration of $\dcl(Ca_w)$ without repetitions.
Define $s_w(y_w^0)$ to be the type $\tp(b_w/C)$.
We may assume that the tuple of variables $y_w^0$ used to denote $b_w$ contains $x_w$ in such a way that that $b_w \rest x_w = a_w$ and that all the other variables are completely new.
We call $\rho$ the bijection mapping a variable in $y_w^0$ to the corresponding element in $b_w$.

\medskip
\noindent
We want to define a bijective renaming $\sigma: y_w^0\to y_w$ of $y_w^0$ fixing $x_w$ with the following properties:
\begin{enumerate}
	\item For all previously constructed $y_u$ we have $y_v\cap y_w=y_{v\cap w}$.
	\item If $v$ is strictly contained in $w$, then $s_w(y_w)\rest y_v=s_v(y_v)$.
	\item For all $w'$ in $W$ we have $y_w \cap x_w'= x_{w\cap w'}$
\end{enumerate}

\noindent
Let $z$ be a variable in $y_w^0$ and set $b=\rho(z)$.
If $b$ does not belong to any of the sets $\dcl(C a_v), v<w$, we leave $z$ untouched, i.e. we set $\sigma(z)=z$.
Otherwise there is a $C$-definable map $f$ such that $b$ is equal to $f(a_v)$ for some $v<w$.
Since $s_v(y_v)$ is an enumeration of the definable closure of a realization of $q_v(x_v)$, there exists a variable $z'$ in $y_v$ such that the formula $"z'=f(x_v)"$ belongs to $s_v(y_v)$.
We then set $\sigma(z)=z'$, and we define $y_w$ as the image of $y_w^0$ under $\sigma$.

Notice that if $z$ belongs to $x_w$, then $\sigma(z)=z$, i.e. $x_w$ is fixed under $\sigma$.
To see this, we may assume that $\rho(z)$ belongs to $\acl(Ca_v)$ for some $v\subsetneq w$ (otherwise the result is immediate).
We know that $\rho(x_w)=a_w$ and that $a_w\cap \acl(Ca_v)=a_v$ by hypothesis, hence $z$ belong to $x_v$.
Now applying the process to any variable of $x_v$ replaces it by itself, therefore we get $\sigma(z)=z$, i.e. $\sigma$ fixes $x_w$.
More generally, for any $w'$ in $W$, we have $y_w\cap x_{w'} = x_{x\cap w'}$ so (3) holds.

\bigskip
\noindent
CLAIM: The map $\sigma: y_w^0 \to y_w$ is well-defined.
\renewcommand\qedsymbol{$\square_{\text{Claim}}$}

\begin{proof}
Remark first that $z'$ does not depend on the choice of $f$: if
$b=f(a_v)=g(a_v)$, then the formula $"f(x_v)=g(x_v)"$ belongs to $\tp(a_v/C)=q_v(x_v)$.
It follows that if $s_v$ contains both $"z'=f(x_v)"$ and $"z''=g(x_v)"$, then $s_v$ also contains $"z'=z''"$.
Now $s_v$ is the type of an enumeration without repetitions, whence $z'$ and $z''$ are actually the same variable.

Moreover, this operation does not depend on the choice of $v$.
To see this, remark that if $b$ is definable over both $Ca_u$ and $Ca_v$, then, by independence of $a_u$ and $a_v$ over $Ca_{u\cap v}$, we also get that $b$ belongs to $\acl(Ca_{u\cap v})$, which is equal to $\dcl(Ca_{u\cap v})$ by hypothesis.
This yields a $C$-definable map $f$ such that $b$ equals $f(a_{u\cap v})$.
In particular, we have variables $z_1, z_2$ and $z_3$ such that the formulae 
$"z_1=f(x_{u\cap v})", "z_2=f(x_{u\cap v})"$ and $"z_3=f(x_{u\cap v})"$ belong to $s_u, s_v$ and $s_{u\cap v}$ respectively.

By induction hypothesis,  $s_{u\cap v}$ is contained into both $s_u$ and $s_v$, and $"z_3=f(x_{u\cap v})"$ belongs therefore to both $s_u$ and $s_v$, so we have $z_1=z_3=z_2$ by the discussion above.

Note that this reasoning also proves the injectivity (and hence bijectivity) of $\sigma$. 
\end{proof}
\renewcommand\qedsymbol{$\square$}

Set $s_w(y_w)$ to be the image of $s_w(y_w^0)$ under $\sigma$.
By construction $b_w$ realizes $s_w(y_w)$ and $b_w\rest y_v$ realizes $s_v(y_v)$ for any $v$ strictly contained in $w$, hence $s_w(y_w)\rest y_v=s_v(y_v)$, i.e. (2) holds.
Let us now see that if $y_u$ had previously been constructed, then $y_u\cap y_w=y_{u\cap  w}$.
By induction hypothesis we only have to show the inclusion $y_u\cap y_w\sseq y_{u\cap  w}$.
Let $z$ be a variable in $y_u\cap y_w$.
We already know that $y_u \cap x_w= x_{u\cap w}$, so there are two remaining possibilities:
\begin{itemize}
	\item If $z$ belongs to some $y_v, v \subsetneq w$, then $z$ belongs to $y_u\cap y_v$, which is equal to $y_{u\cap v}$ by induction hypothesis.
	Now $y_{u\cap v}$ is contained in $y_{u\cap w}$ again by induction hypothesis (since $|u\cap w| < |w|$), so $z$ belongs to $y_{u\cap w}$.
	
	\item Otherwise $z$ is a new variable and is fixed by $\sigma$. In particular, the variable $z$ does not belong to $y_u$ so this case does not happen.
\end{itemize}

It follows that the system obtained after adding $s_w(y_w)$ is again an amalgamation system satisfying the desired properties, which concludes the proof.
\end{proof}

An immediate consequence of this result for theories $T$ with weak elimination of imaginaries is that $T\eq$ has $n$-existence over $\acleq(C)$ if and only if $T$ has $n$-existence over $C$, for $C$ a real algebraically closed set.

\medskip

Finally, Hrushovski proved in \cite{hrushovski_2012} that if a (stable) theory $T$ has $n$-uniqueness over all algebraically closed sets, then it also has $(n+1)$-existence over algebraically closed sets.
It turns out that the hypothesis can be weakened, as the following lemma proves that $n$-uniqueness over an algebraically closed set $C$ implies $(n+1)$-existence over $C$.

\begin{lemma} \label{lemma:UniquenessImpliesExistence}
If $T$ has $n$-uniqueness and $2$-existence over an algebraically closed set $C$, then $T$ has $(n+1)$-existence over $C$.
\end{lemma}

\begin{proof}
Let $S=(p_w, w\in W)$ be an algebraically closed $(n+1)$-amalgamation system.
We may assume that $W= \mathfrak{P}([n])^-$.
For all $k\leq n$, let $W_k=\{w\in W|\: |w|\le k\}$ and $S_k= (p_w, w\in W_k)$. We show by induction on $k$ that $S_k$ has a completion, which will yield the result by taking $k=n$.

For $k=1$, use $2$-existence to find an independent family $a_1,...a_{n+1}$ such that $a_i$ is a realization of $p_i$, for all $i\leq n+1$.
Then $\tp(a_1,...a_{n+1}/C)$ is a completion of $S_1$.

Assume now that $k < n$ is such that $S_k$ has a completion and let $a_\top = \bigcup_{w\in W_k} a_k$ be a realization of this completion.
Let $w$ be an element of $W_{k+1} \setminus W_k$.
Then $(p_v, v\subsetneq w)$ is a $|w|$-amalgamation system over $C$, and both $p_w(x_w) \rest \bigcup_{v\subsetneq w} x_v$ and $\tp(a_v, v \subsetneq w/C)$ are completions of this system.
By $n$-uniqueness and since $|w|\le n$, we can deduce that $\bigcup_{v\subsetneq w} a_v$ realizes $p_w(x_w) \rest \bigcup_{v\subsetneq w} x_v$.
In particular, there is an enumeration $a_w$ on $\acl(a_i, i\in w)$ expanding the one on  $\bigcup_{v\subsetneq w} a_v$ such that $a_w$ realizes $p_w$.
Moroever if $w\neq w'$, then we have $\acl(a_i, i\in w) \cap \acl(a_i, i\in w') = \acl(a_i, i\in w\cap w')$, so the enumerations $a_w$ and $a_{w'}$ are compatible on their intersection $a_{w \cap w'}$.
It follows that $\tp(a_w, w\in W_k /C)$ is a completion of $S_k$.
\end{proof}

\smallskip

\subsection{Relative uniqueness} \label{sec:RelUniq}
We provide here a generalization of the proof of \cite[Proposition 1.6]{DePiro_Kim_Millar_2005} (see Fact \ref{Sharp}) to the case of a theory in which types over algebraically closed sets are stationary (but which does not necessarily eliminates imaginaries).

We highlight the role of $n$-uniqueness, as it will be used in section \ref{sec:FieldsOperators} to prove that some unstable expansions of fields have $n$-existence.
Moreover, we formulate the results relatively to a subset of a model of the theory in order to obtain a lemma used in section \ref{sec:PAC} to prove $n$-existence in bounded PAC substructures of stable theories.

\medskip

In this subsection, we fix a $|T|^+$-saturated model $M$ of a stable theory $T$ such that all the types over algebraically closed sets are stationary (i.e. the theory $T$ has $2$-uniqueness over lagebraically closed sets).
Given a subset $P$ of $M$ and a subset $A$ of $P$, we say that $A$ is \textit{relatively algebraically closed (in $P$)} if 
$$A=\acl(A) \cap P.$$
We assume that $P$ is a subset of $M$ with the following properties:

\begin{property}
\begin{enumerate}
	
	\item  \label{smallhyp:RelExtension}
	(\textit{relative extension}) For all $a$ and $b$ in $P$ and $C$ relatively algebraically closed in $P$, there is $a'$ in $P$ with the same type as $a$ over $C$ such that $a'\ind_C b$ holds.
	
	\item \label{smallhyp:RelClosure}
	(\textit{relative algebraic closure}) For all $a$ and $a'$ in $P$ such that $a\equiv a'$, if $a$ is relatively algebraically closed, then so is $a'$.

	\item \label{smallhyp:StrongBdd}
	The subset $P$ is \emph{strongly bounded} in $M$, i.e. for every relatively algebraically closed $A$ in $P$, we have that
\[
	\acl(A)= \dcl(A\acl(\emptyset)).
\]
Note that this implies that for all subset $B$ of $A$, we have that 
\[
	\acl(A)= \dcl(A\acl(B)).
\]
\end{enumerate}

\end{property}

The next property will be used again in Section \ref{sec:TransferNExistence}. We assume that it holds for the rest of the section.

\begin{property}[$\star$] 
\phantomsection
\label{Star}

For any tuple $a$ in $P$ and every subset $C$ of $P$ such that $C$ is relatively algebraically closed in $P$, the type $\tp(a/C)$ is stationary. 

\end{property}

We immediately see that if $P$ coincide with the whole structure $M$, then all of these properties are satisfied.
In particular, the results of the rest of the section hold when taking $P=M$.
The other example of such $P$ that we will use is that of a bounded PAC substructure of $M$ (see Section \ref{sec:PAC}).

\begin{remark} \label{rk:ApplicationOfStar}
If $a$ and $b$ are tuples in $P$ and $C$ is a relatively algebraically closed subset of $P$, then $\tp(ab/C)$ is stationary by Property (\nameref{Star}).
In particular the type $\tp(ab/C)$ entails $\tp(ab/\acl(C))$, from which we get 
\[\tp(a/bC) \vdash \tp(a/b\acl(C)).\]
\end{remark}

We use a result present in \cite[Lemma 3.17]{Polkowska_2007} under the name "coheir lemma"; the version we present here is due to Blossier and Martin-Pizarro in \cite[Lemme 1.12]{Blossier_MartinPizarro_2019}.
This version is stronger, as it allows to consider types over existentially closed sets containing an elementary substructures and not just elementary substructures.

\begin{fact}{ \cite[Lemme 1.12]{Blossier_MartinPizarro_2019}} \label{CoheirLemma}
If $Q$ is an $\mcL$-existentially closed substructure of $P$, then for all tuples $a$ and $b$ from $P$ independent over $Q$, the type $\tp(a/\acl(Qb)$ is coheir over $Q$.
\end{fact}

Note that the converse is also true, i.e. the subsets of $P$ satisfying the conclusion of Fact \ref{CoheirLemma} are exactly the existentially closed substructures of $P$.

The proof of the next lemma comes from the proofs of $n$-existence in \cite[Theorem 2.1]{HrushovskiPFFieldsRelStructures} and of the independence theorem in \cite[Theorem 3.17]{Polkowska_2007}; it shows a useful entailement of types under some independence assumptions.

\begin{lemma}	\label{lemma:SharpEntailment}
Fix an existentially closed subset $Q$ of $P$.
Let $a_1,..., a_n$ be relatively algebraically closed supersets of $Q$ and $b$ a tuple from $P$ independent from $a_1...a_n$ over $Q$.
If $c$ is algebraic over $a_1...a_n$, then 
\[
	\tp(c/a_1...a_n) \vdash \tp(c/\acl(a_1b)...\acl(a_nb)).
\]
\end{lemma}

\begin{proof}
By Remark \ref{rk:ApplicationOfStar}, we see that $\tp(c/a_1...a_n)$ entails $\tp(c/a_1...a_n \acl(Q))$, which in turns entails the type $\tp(c/\acl(a_1)...\acl(a_n))$ by Property (\ref{smallhyp:StrongBdd}).
It follows that we only have to show the entailment of types
\[
	\tp(c/\acl(a_1)...\acl(a_n)) \vdash \tp(c/\acl(a_1b)...\acl(a_nb)).
\]
We can find a formula $\phi(x)$ isolating the type of $c$ over $\acl(a_1)...\acl(a_n)$ and a formula $\psi(x, d_1...d_n)$ isolating the type of $c$ over $\acl(a_1b)...\acl(a_nb)$.
Take then the $\theta_i(y_i, a_i, b)$'s to be formulae witnessing the algebraicity of $d_i$ over $a_ib$.
Assume by contradiction that the result of the lemma is not satisfied, i.e. that $\phi(x)$ is strictly contained in $\psi(x, d_1...d_n)$.
The formula 
$$\rho(z) = \exists y_1...y_n, \psi(c, y_1...y_n) \wedge "\psi(x, y_1...y_n) \subsetneq \phi(x)"\wedge \bigwedge_{i\leq n} \theta_i(y_i, a_i, z) $$ 
then belongs to the type of $b$ over $\acl(a_1...a_n)$.
This implies, by Fact \ref{CoheirLemma}, that $\rho$ has a realization $b'$ contained in $Q$.
We then obtain elements $d_1',...,d_n'$ with $d_i$ belonging to $\acl(a_iQ)=\acl(a_i)$,  such that $\psi(c, d_1'...d_n')$ holds and $\psi(x, d_1'...d_n')$ is strictly contained in $\phi(x)$.
This contradicts the fact that $\phi(x, a_1...a_n)$ isolates the type of $c$ over $\acl(a_1)...\acl(a_n)$.
\end{proof}

\begin{remark}
The conclusion of Lemma \ref{lemma:SharpEntailment} implies
\[
\dcl(\acl(a_1b)...\acl(a_nb))\cap \acl(a_1...a_n) = \dcl(\acl(a_1)...\acl(a_n))
\]
i.e. the conclusion of property $(\sharp_n)$ of Fact \ref{Sharp} holds.
Conversely, if $T$ has elimination of imaginaries, the set $P$ is definably closed (or more generally if every element of $\dcl(P)$ is interdefinable with a tuple of elements from $P$) and the conclusion of $(\sharp)_n$ holds for $a_1,...,a_n$ and $b$ in $P$, then so does the conclusion of Lemma \ref{lemma:SharpEntailment}, as shown in \cite[Proposition 1.5]{DePiro_Kim_Millar_2005}.
\end{remark}

We say that an amalgamation system $(p_w(x_w), w\in W)$ over a relatively algebraically closed subset $Q$ of $P$ is \textit{algebraically closed relatively to} $P$ if every $p_w$ has a realization which is relatively algebraically closed in $P$.
In such a system, Property (\nameref{Star}) yields that the $p_w$'s are stationary.

\begin{proposition} \label{prop:RelNUniq}
Let $Q$ be an existentially closed subset of $P$.
Fix  an amalgamation system $S=(p_w(x_w), w\in W)$ over $Q$ which is algebraically closed relatively to $P$.
If $S$ has a minimal completion, then this completion is unique and stationary.
\end{proposition}

\begin{proof}
We assume that $W= \mathfrak{P}([n])^-$.
Let us introduce some notations: 
for $2\leq k \leq n+1$, we denote 
$\Delta_k= \{w\in W| [k,n] \sseq w \text{ and } |w|=n-1\}$ and $\Gamma_k= \{ w\in W| [k+1,n] \sseq w, k\notin w \text{ and } |w|=n-2 \}$.
Moreover, if $a_\top$ is a realization of a completion of $S$ and $V$ a subset of $W$, set $a_V= \bigcup_{v\in V} a_v$.
In particular, we have
\[
	a_{\Delta_k} = a_{[n]\setminus \{1\}}... a_{[n]\setminus \{k-1\}}
	\text{ and }
	a_{\Gamma_k} = a_{[n]\setminus \{1, k\}}... a_{[n]\setminus \{k-1, k\}}
\]
\noindent
Remark also that we have  $a_{\Delta_{k+1}}= a_{[n]\setminus \{k\}} a_{\Delta_k}$, 
that $a_{\Gamma_k}\sseq a_{\Delta_k}\cap a_{[n]\setminus \{k\}}$,
and that  $a_{\Delta_3} = a_{\{1,3,...,n\}} a_{\{2,3,...,n\}}$.

\noindent
We show by induction on $k\geq 3$ that if $a_\top$ and $b_\top$ are such that their types over $Q$ are completions of $S$, then $a_{\Delta_k}$ and $b_{\Delta_k}$ have the same type over $\acl(Q)$; this will prove the result as we have $a_{\Delta_{n+1 }}= a_\top$ and types over algebraically closed sets are stationary.

\begin{example}
We represent the different cases for $n=4$, where we circled the set $a_{[n]\setminus\{k\}}$, boxed the set $a_{\Delta_k}$ and boxed with rounded angles the set $a_{\Gamma_k}$.

\medskip
\begin{center}
\begin{tikzpicture}

\draw[draw=black] (4.7,2.75) rectangle ++(0.9,0.6);
\draw[fill=none] (3.25,3) circle (0.5) ;
\draw[rounded corners] (6.3, 1.2) rectangle (7.1,1.8);

\node (A0) at (2.4,-1.3) {$a_\emptyset$};

\node (A1) {$a_1$};
\node (A2) [right= of A1] {$a_2$};
\node (A3) [right= of A2] {$a_3$};
\node (A4) [right= of A3] {$a_4$};

\node (A12) at (-2,1.5) {$a_{12}$};
\node (A13) [right= of A12] {$a_{13}$};
\node (A23) [right= of A13] {$a_{23}$};
\node (A14) [right= of A23] {$a_{14}$};
\node (A24) [right= of A14] {$a_{24}$};
\node (A34) [right= of A24] {$a_{34}$};

\node (A123) at (-0.5,3) {$a_{123}$};
\node (A124) [right= of A123] {$a_{124}$};
\node (A134) [right= of A124] {$a_{134}$};
\node (A234) [right= of A134] {$a_{234}$};

\draw [->, dashed] (A0) -- (A1) node  {};
\draw [->, dashed] (A0) -- (A2) node  {};
\draw [->] (A0) -- (A3) node  {};
\draw [->] (A0) -- (A4) node  {};

\draw[->, dashed] (A1) -- (A12) node {};
\draw[->, dashed] (A1) -- (A13) node {};
\draw[->, dashed] (A1) -- (A14) node {};
\draw[->, dashed] (A2) -- (A12) node {};
\draw[->, dashed] (A2) -- (A23) node {};
\draw[->, dashed] (A2) -- (A24) node {};
\draw[->, dashed] (A3) -- (A13) node {};
\draw[->, dashed] (A3) -- (A23) node {};
\draw[->] (A3) -- (A34) node {};
\draw[->, dashed] (A4) -- (A14) node {};
\draw[->, dashed] (A4) -- (A24) node {};
\draw[->] (A4) -- (A34) node {};

\draw[->, dashed] (A12) -- (A123) node {};
\draw[->, dashed] (A13) -- (A123) node {};
\draw[->, dashed] (A23) -- (A123) node {};
\draw[->, dashed] (A12) -- (A124) node {};
\draw[->, dashed] (A14) -- (A124) node {};
\draw[->, dashed] (A24) -- (A124) node {};
\draw[->, dashed] (A13) -- (A134) node {};
\draw[->, dashed] (A14) -- (A134) node {};
\draw[->] (A34) -- (A134) node {};
\draw[->, dashed] (A23) -- (A234) node {};
\draw[->, dashed] (A24) -- (A234) node {};
\draw[->] (A34) -- (A234) node {};
\end{tikzpicture}
\end{center}
\nopagebreak
\begin{center}
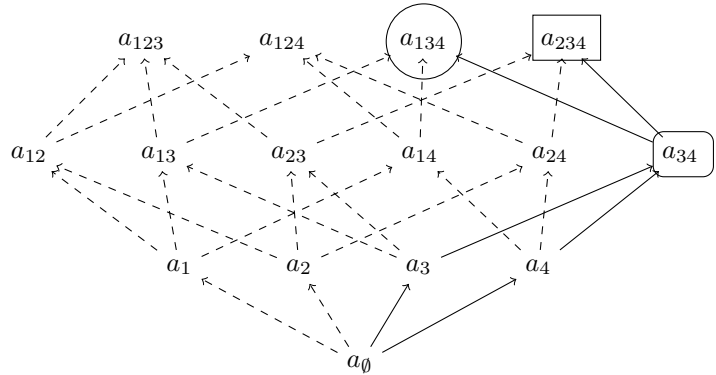
{The case $k=2$}
\end{center}
\medskip

\begin{center}
\begin{tikzpicture}

\draw[draw=black] (2.7,2.75) rectangle ++(2.9,0.6);
\draw[fill=none] (1.4,3) circle (0.5) ;
\draw[rounded corners] (2.75, 1.2) rectangle (5.3,1.8);

\node (A0) at (2.4,-1.3) {$a_\emptyset$};

\node (A1) {$a_1$};
\node (A2) [right= of A1] {$a_2$};
\node (A3) [right= of A2] {$a_3$};
\node (A4) [right= of A3] {$a_4$};

\node (A12) at (-2,1.5) {$a_{12}$};
\node (A13) [right= of A12] {$a_{13}$};
\node (A23) [right= of A13] {$a_{23}$};
\node (A14) [right= of A23] {$a_{14}$};
\node (A24) [right= of A14] {$a_{24}$};
\node (A34) [right= of A24] {$a_{34}$};

\node (A123) at (-0.5,3) {$a_{123}$};
\node (A124) [right= of A123] {$a_{124}$};
\node (A134) [right= of A124] {$a_{134}$};
\node (A234) [right= of A134] {$a_{234}$};

\draw [->] (A0) -- (A1) node  {};
\draw [->] (A0) -- (A2) node  {};
\draw [->, dashed] (A0) -- (A3) node  {};
\draw [->] (A0) -- (A4) node  {};

\draw[->, dashed] (A1) -- (A12) node {};
\draw[->, dashed] (A1) -- (A13) node {};
\draw[->] (A1) -- (A14) node {};
\draw[->, dashed] (A2) -- (A12) node {};
\draw[->, dashed] (A2) -- (A23) node {};
\draw[->] (A2) -- (A24) node {};
\draw[->, dashed] (A3) -- (A13) node {};
\draw[->, dashed] (A3) -- (A23) node {};
\draw[->, dashed] (A3) -- (A34) node {};
\draw[->] (A4) -- (A14) node {};
\draw[->] (A4) -- (A24) node {};
\draw[->, dashed] (A4) -- (A34) node {};

\draw[->, dashed] (A12) -- (A123) node {};
\draw[->, dashed] (A13) -- (A123) node {};
\draw[->, dashed] (A23) -- (A123) node {};
\draw[->, dashed] (A12) -- (A124) node {};
\draw[->] (A14) -- (A124) node {};
\draw[->] (A24) -- (A124) node {};
\draw[->, dashed] (A13) -- (A134) node {};
\draw[->] (A14) -- (A134) node {};
\draw[->, dashed] (A34) -- (A134) node {};
\draw[->, dashed] (A23) -- (A234) node {};
\draw[->] (A24) -- (A234) node {};
\draw[->, dashed] (A34) -- (A234) node {};

\end{tikzpicture}
\end{center}
\nopagebreak
\begin{center}
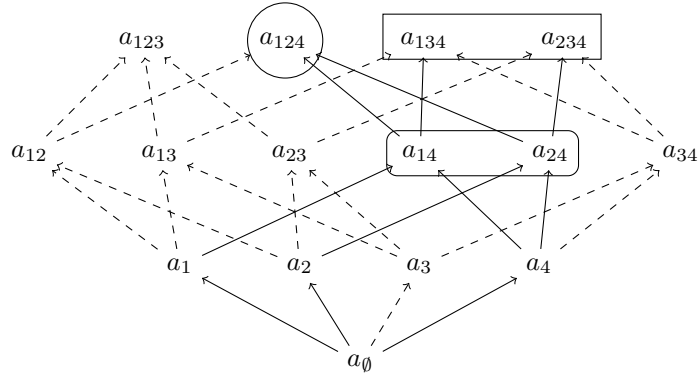
{The case $k=3$}
\end{center}
\bigskip

\begin{center}
\begin{tikzpicture}

\draw[draw=black] (0.85,2.75) rectangle ++(4.8,0.6);
\draw[fill=none] (-0.5,3) circle (0.5) ;
\draw[rounded corners] (-2.5, 1.2) rectangle (2,1.8);

\node (A0) at (2.4,-1.3) {$a_\emptyset$};

\node (A1) {$a_1$};
\node (A2) [right= of A1] {$a_2$};
\node (A3) [right= of A2] {$a_3$};
\node (A4) [right= of A3] {$a_4$};

\node (A12) at (-2,1.5) {$a_{12}$};
\node (A13) [right= of A12] {$a_{13}$};
\node (A23) [right= of A13] {$a_{23}$};
\node (A14) [right= of A23] {$a_{14}$};
\node (A24) [right= of A14] {$a_{24}$};
\node (A34) [right= of A24] {$a_{34}$};

\node (A123) at (-0.5,3) {$a_{123}$};
\node (A124) [right= of A123] {$a_{124}$};
\node (A134) [right= of A124] {$a_{134}$};
\node (A234) [right= of A134] {$a_{234}$};

\draw [->] (A0) -- (A1) node  {};
\draw [->] (A0) -- (A2) node  {};
\draw [->] (A0) -- (A3) node  {};
\draw [->, dashed] (A0) -- (A4) node  {};

\draw[->] (A1) -- (A12) node {};
\draw[->] (A1) -- (A13) node {};
\draw[->, dashed] (A1) -- (A14) node {};
\draw[->] (A2) -- (A12) node {};
\draw[->] (A2) -- (A23) node {};
\draw[->, dashed] (A2) -- (A24) node {};
\draw[->] (A3) -- (A13) node {};
\draw[->] (A3) -- (A23) node {};
\draw[->, dashed] (A3) -- (A34) node {};
\draw[->, dashed] (A4) -- (A14) node {};
\draw[->, dashed] (A4) -- (A24) node {};
\draw[->, dashed] (A4) -- (A34) node {};

\draw[->] (A12) -- (A123) node {};
\draw[->] (A13) -- (A123) node {};
\draw[->] (A23) -- (A123) node {};
\draw[->] (A12) -- (A124) node {};
\draw[->, dashed] (A14) -- (A124) node {};
\draw[->, dashed] (A24) -- (A124) node {};
\draw[->] (A13) -- (A134) node {};
\draw[->, dashed] (A14) -- (A134) node {};
\draw[->, dashed] (A34) -- (A134) node {};
\draw[->] (A23) -- (A234) node {};
\draw[->, dashed] (A24) -- (A234) node {};
\draw[->, dashed] (A34) -- (A234) node {};

\end{tikzpicture}
\end{center}
\nopagebreak
\begin{center}
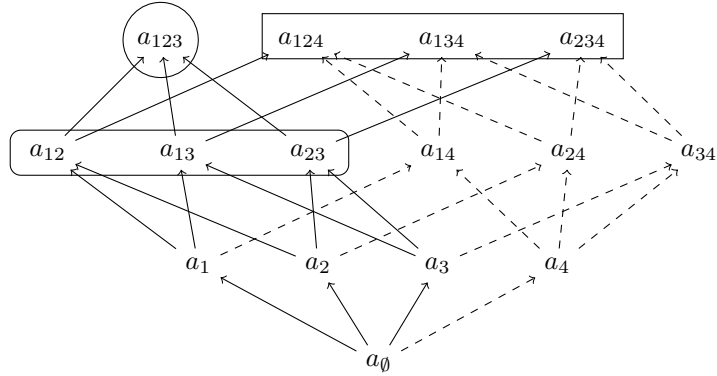
{The case $k=4$}
\end{center}
\end{example}

\bigskip

\noindent
Assume that $\tp(a_\top/Q)$ and $\tp(b_\top/Q)$ are two completions of of $S$.
Note that $a_{\{2,3,...,n\}}$ and $b_{\{2,3,...,n\}}$ have the same type over $Q$, so we may assume that they are equal.

Let $c_{\{1,3,...,n\}}$ be a realization of $p_{\{1,3,...,n\}}$ contained in $P$.
Then $c_{\{3,...,n\}}$ is relatively algebraically closed by Property (\ref{smallhyp:RelClosure}), whence $\tp(c_{\{1,3,...,n\}}/c_{\{3,...,n\}})$ is stationary by Property (\nameref{Star}).
It follows that the same holds for any realization of $p_{\{1,3,...,n\}}$; in particular the type $\tp(a_{\{1,3,...,n\}}/a_{\{3,...,n\}})$ is stationary.
Since $b_{\{1,3,...,n\}}$ has the same type as $a_{\{1,3,...,n\}}$ over $a_{\{3,...,n\}}$ and both are independent from $a_{\{2,3,...,n\}}$ over $a_{\{3,...,n\}}$, we get that they have the same type over $a_{\{3,...,n\}}$, and that this type is stationary.
Now the type $\tp(a_{\{3,...,n\}}/Q)$ is also stationary, so we obtain that $a_{\Delta_3}$ and $b_{\Delta_3}$ have the same type over $Q$ and that this type is stationary.

\medskip

Assume now by induction that $a_{\Delta_k}$ and $b_{\Delta_k}$ have the same type over $\acl(Q)$, for $3\leq k \leq n$. 
We may in particular assume that they are equal; let us show that $a_{[n]\setminus \{k\}}$ and $b_{[n]\setminus \{k\}}$ have the same type over $a_{\Delta_k} \acl(Q)$.

We may assume that both $a_{[n]\setminus \{k\}}$ and $a_{k}$ are contained in $P$.
To see this, recall that both  $p_{[n]\setminus \{k\}}$ and $p_{k}$ have realizations which are contained in $P$ by hypothesis on the system.
We can then use Property (\ref{smallhyp:RelExtension}) to find realizations $a_{[n]\setminus \{k\}}'$ and $a_{k}'$ of $p_{[n]\setminus \{k\}}$ and $p_{k}$ respectively which are both contained in $P$ and are independent over $Q$.
Now $p_{[n]\setminus \{k\}}$ and $p_{k}$ are stationary, whence $a_{[n]\setminus \{k\}}'a_{k}'$ has the same type as $a_{[n]\setminus \{k\}}a_{k}$: replacing $a_\top$ by its image under an automorphism mapping $a_{[n]\setminus \{k\}}a_{k}$ to $a_{[n]\setminus \{k\}}'a_{k}'$ yields the result.

\smallskip

We know that $a_{[n]\setminus \{k\}}$ and $b_{[n]\setminus \{k\}}$ are both realizations of $p_{[n]\setminus \{k\}}$, so they have the same type over $\acl(Q)$.
It follows that  $a_{[n]\setminus\{k\}}a_{\Gamma_k} \equiv_{\acl(Q)} b_{[n]\setminus\{k\}} b_{\Gamma_k}$, since $\Gamma_k$ is contained in ${[n]\setminus\{k\}}$.
Now $\Gamma_k$ is also contained in $\Delta_k$ and we assumed $a_{\Delta_k}=b_{\Delta_k}$, so we in fact have $a_{[n]\setminus\{k\}} \equiv_{a_{\Gamma_k} \acl(Q)} b_{[n]\setminus\{k\}}$.

Remark now that $a_{[n]\setminus\{k\}}$ is algebraic over $a_{\Gamma_k}$ (because $k\geq 3$) and that $a_{\Gamma_k}$ is independent from $a_k$ over $C$.
Moreover we have $a_{\Gamma_k} = (a_v, v\in \Gamma_k)$ and $a_{\Delta_k}=(a_w, w\in \Delta_k) = (\acl(a_va_k), v\in \Gamma_k)$.
We can therefore apply Lemma \ref{lemma:SharpEntailment}, which yields
\[
	\tp(a_{[n]\setminus\{k\}}/a_v, v\in \Gamma_k) \vdash \tp(a_{[n]\setminus\{k\}}/\acl(a_v a_k), v\in {\Gamma_k}).
\]
Recall that both $a_{[n]\setminus \{i\}}$ and $\acl(Q)$ are contained in $\acl(a_{[n]\setminus \{k, i\}}a_k)$ for $i< k$ because the system is controlled, 
so we actually have proven that $\tp(a_{[n]\setminus\{k\}}/a_{\Gamma_k})$ entails $\tp(a_{[n]\setminus\{k\}}/a_{\Delta_k}\acl(Q))$.

Finally, we get that $a_{[n]\setminus\{k\}}$ and $b_{[n]\setminus\{k\}}$ have the same type over $a_{\Delta_k}\acl(Q)$, which concludes the induction step and then the proof.
\end{proof}

In the case where $P$ coincide with $M$, Proposition \ref{prop:RelNUniq} states that $T$ has $n$-uniqueness over existentially closed sets for all $n$.
In particular, we can apply Lemma \ref{lemma:UniquenessImpliesExistence} to obtain the following generalization of Fact \ref{Sharp}:

\begin{cor} \label{cor:nUniqu}
If $T$ is a stable theory such that types over algebraically closed sets are stationary, then $T$ has $n$-uniqueness and $n$-existence over existentially closed subsets for all $n$.
In particular, the theory $T$ has $n$-uniqueness and $n$-existence over elementary substructures.
\end{cor}

\section{Transfer of generalized amalgamation} 	\label{sec:TransferNExistence}

\subsection{The general setting}		\label{sec:MainThm}

We consider a complete simple theory $T_0$ in a language $\mcL_0$ with quantifier elimination (QE).
Take an expansion $\mcL_1$ of $\mcL_0$ such that every symbol in $\mcL'=\mcL_1 \setminus \mcL_0$ is either a constant, a unary relation or a unary function symbol.
Let $T_1$ be a complete simple $\mcL_1$-theory containing the universal part $T_0^\forall$ of $T_0$ (i.e. every model of $T_1$ can be embedded into a model of $T_0$). 
In the following, fix a $\kappa$-saturated model $N$ of $T_1$ contained in a $\kappa$-saturated model $M$ of $T_0$, for some arbitrarily large cardinal $\kappa> |T_1|$.
We say that a set or a tuple is small if it has cardinality strictly smaller than $\kappa$

\smallskip

Given any tuple $a$ of $N$ and any subset $B$ of $N$, we will denote by $\tp_1(a/B)$ the $\mcL_1$-type of $a$ over $B$ in the $\mcL_1$-structure $N$ and by $\tp_0(a/B)$ the $\mcL_0$-type of $a$ over $B$ in the $\mcL_0$-structure $M$.
By extension, any concept indexed by $1$  will refer to its interpretation in the $\mcL_1$-structure $N$, while any concept indexed by $0$ will refer to its interpretation in the $\mcL_0$-structure $M$.

\bigskip

We make several hypotheses on the theories $T_0$ and $T_1$, which we list below.
These hypotheses are all satisfied in the theories we will consider, which will allow us to transfer $n$-existence in all of those contexts.
They were chosen so that the proof of Theorem \ref{thm:AbstractAmalgamationTransfer} works and may therefore seem unnatural; we discuss in Subsection \ref{sec:VerifHyp} some of the more natural ways in which they arise.

\bigskip
\noindent
\textbf{\namedlabel{hyp:StrongBdd}{H1}.} The structure $N$ is \textit{strongly bounded} in $M$, i.e. for all relatively $T_0$-algebraically closed subset $A$ of $N$ we have that
	\[
	\acl_0(A)=\dcl_0(A\acl_0(\emptyset)).
	\]

\begin{remark}
\
\begin{enumerate}
 	\item Conditions similar to this one appear for example in \cite[Proposition 2.5]{Pillay_Polkowska_2006} or in \cite[Hypothèse 4]{Blossier_MartinPizarro_2019}.
	\item This hypothesis holds trivially whenever $N$ is algebraically closed in $M$, as any relatively $T_0$-algebraically closed subset of $N$ is already algebraically closed in $M$.
	
	\item In most of the examples below, this hypothesis will be satisfied after adding parameters for an elementary substructure of $N$ to the theories $T_0$ and $T_1$.	
	
	\item This hypothesis is preserved under adding parameters: if $A$ is relatively $T_0$-algebraically closed in $N$ and $B$ is contained in $A$, then \ref{hyp:StrongBdd} implies that 
		\[\acl_0(AB)= \dcl_0(A\acl_0(B)).\]
	
\end{enumerate}
\end{remark}

\bigskip
\noindent
\textbf{\namedlabel{hyp:CoheirLemmaLight}{H2}.}
If $C$ is a relatively $T_0$-algebraically closed subset of $N$ and $a$ is any tuple from $N$, then the type $\tp_0(a/C)$ has a unique extension to $\acl_0(C)$.

\medskip
\noindent
In practice, there are two main contexts where \ref{hyp:CoheirLemmaLight} holds: if $\acl_0(C)$ equals $\dcl_0(C)$ (which is in particular the case when $C=\acl_0(C))$, or if the type $\tp_0(a/C)$ is stationary.

\medskip 
\noindent
We assume that there is a class $\mcG$ of small $\mcL'$-substructures of $N$,  preserved by $\mcL_1$-automorphisms of $N$, such that the following hypothesis hold:

\bigskip
\noindent
\textbf{\namedlabel{hyp:InteractionAlgClosure}{H3}.}
If $a$ is a $T_1$-algebraically closed subset of $N$, then $a$ belongs to $\mcG$.
\\
Moreover, if $a$ belongs to $\mcG$, then
\[
	\acl_1(a)=\acl_0(a) \cap N.
\]

\medskip
\noindent
We will also require a characterization of the independence relation $\ind^1$ in terms of $\ind^0$ on the class $\mcG$. 

\bigskip
\noindent
\textbf{\namedlabel{hyp:CaracIndep}{H4}.}
For all $A,B$ and $C$ in $\mcG$ with $C\sseq A\cap B$, we have that
\[
	A \ind_C^1 B \Leftrightarrow A\ind^0_C B  \text{ and } AB \text{ belongs to } \mcG.
\]

\medskip

Remark that if $\mcG$ is closed under union, then this characterization is equivalent to saying that $\ind^1$ and $\ind^0$ coincide for elements of $\mcG$.
There are also known examples of theories in which the independence follows such a characterization with a nontrivial class $\mcG$, such as the theories of lovely pairs, of $H$-structures or of $\kappa$-PAC beautiful pairs: we will study these theories with more details is Section \ref{sec:PairsOfStructures}.

\bigskip
\noindent
\textbf{\namedlabel{hyp:SameType}{H5}.}
For all $a$ and $b$ tuples of $N$ such that $a$ belongs to $\mcG$ and $b$ is algebraically closed in $T_1$, if $a$ and $b$ have the same quantifier-free $\mcL_1$-type, then they have the same $\mcL_1$-type.

\bigskip

Note that if $a$ is a $\mcL_1$-substructure of $N$, then $\qftp_1(a)$ can be identified with the pair $(\tp_0(a), \qftp'(a))$.
More generally, we define an \textit{annotation} on a (quantifier-free) $\mcL_0$-type $p(x)$ as the quantifier-free $\mcL'$-type $I(x)$ of a $\mcL'$-structures in the variable $x$, and we say that the pair $(p(x), I)$ is an \textit{annotated type}.
If $p(x)\cup I(x)$ has a realization $a$ belonging to $\mcG$ we then say that $(p,I)$ is a \textit{fair type}, and that it is a \textit{relatively algebraically closed fair type} if we can choose $a$ to be  algebraically closed in $T_1$. 

Two annotated types $(p_1(x_1), I_1)$ and $(p_2(x_2), I_2)$ are said to be compatible if $p_1(x_1) \rest x_1\cap x_2 = p_2(x_2) \rest x_1\cap x_2$ and $I_1(x_1) \rest x_1\cap x_2= I_2(x_2) \rest x_1\cap x_2$.
In that case, if $p(x_1x_2)$ is a completion of $p_1(x_1) \cup  p_2(x_2)$, then $(p(x_1x_2), I_1I_2)$ is again an annotated pair by choice of the language $\mcL'$.

Analogously, we define an \textit{annotated amalgamation system} as a family  of annotated types $(p_w(x_w),  I_w)_{w\in W}$ such that $(p_w(x_w))_{w\in W}$ is an amalgamation system in the sense of $T_0$ and the family $(I_w)_{w\in W}$ is compatible (i.e. for all $v$ and $w$ in $W$, we have $I_v\cap I_w = I_{v\cap w}$).
It is a \textit{relatively algebraically closed fair system} if all the $p_w$'s are  relatively algebraically closed fair types, and a \textit{strongly fair system} if for all $w$ in $W$, the type $p_w$ has a realization $a_w$ such that for all $v_1,...,v_m$ contained in $w$, the union  $a_{v_1}...a_{v_m}$ belongs to $\mcG$.
Finally, we say that it is a \textit{relatively algebraically closed strongly fair system} if it is both a relatively algebraically closed fair system and a strongly fair system.

\begin{remark}
If $(p_w(x_w), I_w)_{ w\in W}$ is a relatively algebraically closed strongly fair system, then we see by \ref{hyp:SameType} that every realization of $p_w$ witnessing the strongly fair system property also witnesses the relatively algebraically closed fair system property, i.e. any realization $a_w$ of $p_w\cup I_w$  belonging  $\mcG$ is algebraically closed in $T_1$.
\end{remark}

\bigskip
\noindent
\textbf{\namedlabel{hyp:ContinuityFSystems}{H6}.}
Every minimal completion of a relatively algebraically closed strongly fair  system is a strongly fair system.
\\
In other words, if $(p_w(x_w), I_w)_{ w\in W}$ is a relatively algebraically closed strongly fair system and $p_\top(x_\top)$ is a minimal completion of $(p_w(x_w), w\in W)$, then the system $(p_w(x_w), I_w)_{w\in W} \cup \{(p_\top(x_\top), \bigcup_{w\in W} I_w)\}$ is a strongly fair system.

\smallskip

\begin{remark}
\
\begin{enumerate}
	\item If $p_\top(x_\top)$ is a minimal completion of an algebraically closed amalgamation system and $a_\top$ a realization of $p_\top$, then $a_\top= \bigcup_{w\in W} a_w$ is a union of algebraically closed sets but has, in the general case, no reason to be algebraically closed itself.
	In particular, the completed system is not necessarily algebraically closed.
	
	\item In the concrete theories we will consider in the following sections, hypotheses \ref{hyp:StrongBdd}-\ref{hyp:SameType} will be obtained as consequences of already known results, leaving \ref{hyp:ContinuityFSystems} as the core of the proof of $n$-existence.
Moreover, proofs of \ref{hyp:ContinuityFSystems} can be very different depending on the theories, which complicates the application of our main result.
\end{enumerate}
\end{remark}

\smallskip

We assume now that all these hypotheses are satisfied.
We can then transfer $n$-existence from the theory $T_0$ to the theory $T_1$, for all integers $n\ge 1$:

\begin{theorem}
\label{thm:AbstractAmalgamationTransfer}
Assume \ref{hyp:StrongBdd}-\ref{hyp:ContinuityFSystems} and let $Q$ be a subset of $N$ which is algebraically closed in $T_1$, as well as $n\ge 1$.
If $T_0$ has $n$-existence over $\acl_0(Q)$, then $T_1$ has $n$-existence over $Q$.
\end{theorem}

\begin{proof}
Let $(r_w(x_w), w\in W)$ be an algebraically closed amalgamation system over $Q$ in $T_1$. 
We may assume $W$ to be the set $\mathfrak{P}([n])^-$.
For all $w$ in $W$, let $p_w(x_w)$ be the restriction of $r_w$ to $\mcL_0$, and let $I_w$ be the restriction of $r_w(x_w)$ to the quantifier-free $\mcL'$-formulae.
Remark that any realization $b_w$ of $r_w$ is algebraically closed in $T_1$, so in particular $b_w$ belongs to $\mcG$ by  \ref{hyp:InteractionAlgClosure}.

\bigskip
\noindent
CLAIM 1: The system $(p_w(x_w), I_w)_{w\in W}$ is a relatively algebraically closed strongly fair system.

\renewcommand\qedsymbol{$\square_{\text{Claim } 1}$}
\begin{proof}
We need to show that $(p_w(x_w), w\in W)$ is an amalgamation system such that  $(p_w(x_w), I_w)_{w\in W}$ is an annotated system satisfying both the relatively algebraically closed fair system and the strongly fair system properties.
Note that the annotated system property is immediate.

\smallskip
Remark  first that if $b_w$ realizes $r_w$, then all the $b_i$'s ($i\in w$) are in $\mcG$ and that they form an independent family over $Q$ in $T_1$. 
This yields by \ref{hyp:CaracIndep} that the family $(b_i, i\in w)$ is independent over $Q$ in $T_0$; this independence is witnessed by the $\mcL_0$-type of $b_w$, i.e. by $p_w$.

Now, \ref{hyp:CaracIndep} also implies that for all $v\sseq w$, the union $\bigcup_{i\in v} b_i$ belongs to $\mcG$.
We can therefore apply \ref{hyp:InteractionAlgClosure}, which yields that $\acl_1(b_i, i\in w)$ is equal to $\acl_0(b_i, i\in w) \cap N$.
Since the system $(r_w, w\in W)$ is controlled, we get that $b_w$ is algebraic over $(b_i, i\in w)$ in $T_0$, which is witnessed by $p_w$.

\smallskip 

The fact that it is a relatively algebraically closed fair system is a consequence of \ref{hyp:InteractionAlgClosure}, as any realization $b_w$ of $r_w$ is also a relatively algebraically closed realization of $p_w$ belonging to $\mcG$.

\smallskip

Finally, to see that it is a strongly fair system, remark that if $b_w$ realizes $r_w$ and $u,v$ are subsets of $w$, then $b_u, b_v$ and $b_{u\cap v}$ belongs to $\mcG$, and $b_u\ind^1_{b_{u\cap v}} b_v$ holds. 
We can then deduce that $b_ub_v$ belongs to $\mcG$ by \ref{hyp:CaracIndep}, and the general case is obtained by induction.
\end{proof}

Now by \ref{hyp:CoheirLemmaLight}, each $p_w(x_w)$ has a unique extension $q_w(x_w)$ to $\acl_0(Q)$, which is necessarily nonforking.
It follows that $(q_w(x_w), w\in W)$ is an amalgamation system over $\acl_0(Q)$ in $T_0$ by Lemma \ref{ExtensionAmalgSystem}.
Remark that if $b_w$ realizes $r_w$, then  $\acl_0(\acl_0(Q)b_w)$ coincides with $\dcl_0(\acl_0(Q) b_w)$, which is witnessed by $p_w$.
It follows that if $a_w$ realizes $q_w$ then $\acl_0(\acl_0(Q)a_w)=\dcl_0(\acl_0(Q)a_w)$, so we can apply Lemma \ref{CompletionLemma} to obtain a corresponding algebraically closed amalgamation system $(s_w(y_w), w\in W)$ over $\acl_0(Q)$.
By choice of $Q$, we can find a completion $s_\top(z)$ of the system.

We can now consider the restriction $p_\top(x_\top)$ of $s_\top(z)$ to the parameters $Q$ and the variables $x_\top$, where $x_\top$ is the reunion of all the  $x_w$, $w\in W$.
It follows that $p_\top(x_\top)$ is a completion of the system $(p_w(x_w), w\in W)$.
We also define $I_\top= \bigcup_{w\in W} I_w$.
The hypothesis \ref{hyp:ContinuityFSystems} yields that the completed system $(p_w(x_w), I_w)_{ w\in W}\cup \{(p_\top, I_\top)\}$ is a strongly fair system: let $c_\top$ be a realization of $p_\top$ which witnesses this fact.

\bigskip
\noindent
CLAIM 2: The type $r_\top=\tp_1(c_\top/Q)$ is a completion of the system $(r_w, w\in W)$ in $T_1$.

\renewcommand\qedsymbol{$\square_{\text{Claim } 2}$}
\begin{proof}
We need to verify that the family $(c_i, i\leq n)$ is independent over $Q$ in $T_1$ and that for every $w$, the type $\tp_1(c_w/Q)$ is equal to $r_w$.

\smallskip
By choice of $c_\top$ we get that $\bigcup_{i\in w} c_i$ belongs to $\mcG$, for any $w\in W$.
Moreover we know that the family $(c_i, i\leq n)$ is independent over $Q$ in $T_0$, because $p_\top$ is a completion of an amalgamation system.
It then follows from \ref{hyp:CaracIndep} that $(c_i, i\leq n)$ is independent over $Q$ in $T_1$.

\smallskip
For the second point, let $b_w$ be a realization of $r_w$.
We know from \ref{hyp:InteractionAlgClosure} that $b_w$ is closed under $\acl_1$, and that $c_w$ belongs to $\mcG$ by choice of $c_\top$. 
Moreover $b_w$ and $c_w$ have the same $\mcL_1$-type over $Q$ by definition of $(p_w, I_w)$ and choice of $c_\top$: it follows that they have the same $\mcL_1$-type over $Q$ by \ref{hyp:SameType}.
\end{proof}
\renewcommand\qedsymbol{$\square$}
\noindent
This concludes the proof of the theorem.
\end{proof}

\subsection{Verification of the hypotheses}		\label{sec:VerifHyp}

The hypotheses \ref{hyp:StrongBdd}-\ref{hyp:ContinuityFSystems} were all used in the proof of Theorem \ref{thm:AbstractAmalgamationTransfer}, but are they also quite strong and may be therefore difficult to verify in practice.
However, there are some contexts in which they can be obtained as consequences of more natural assumptions, which is what we discuss here.

\medskip

We introduced in Section \ref{sec:RelUniq} the property (\nameref{Star}).
This property was already present in several publications as a consequence of the elimination of imaginaries (see \cite{Blossier_MartinPizarro_2019},  \cite{HrushovskiPFFieldsRelStructures}, \cite{Polkowska_2007}); we introduce it independently as a way to bypass a failure of elimination of imaginaries in the theory $T_0$.
Remark that Property (\nameref{Star}) is not a property of the theory $T_0$ alone, but rather of the pair $(M,N)$, and that it is preserved under elementary equivalence.

\medskip

The proof of the following fact in \cite{Blossier_MartinPizarro_2019} assumes gEI, but only uses it to show that Property (\nameref{Star}) holds, and is thus still valid in our context.

\begin{fact} \cite[Remarque 1.10]{Blossier_MartinPizarro_2019} \label{StrongBounded}
Assume that $T_0$ is stable and that $\acl_0(N)$ coincide with $ \dcl_0(N\acl_0(N_0))$.
If Property (\nameref{Star}) holds and every element of $\dcl_0(N)$ is $\mcL_0$-interdefinable with a tuple of elements of $N$,
then for all relatively algebraically closed set $A$ of $N$ containing $N_0$ we have that
	\[\acl_0(A)=\dcl_0(A\acl_0(N_0)).\]

\noindent
In particular, the structure $N$ is strongly bounded in $M$ after adding parameters for $N_0$.
\end{fact}

\smallskip

This fact allows us to deduce \ref{hyp:StrongBdd} from the stability of $T_0$, the property (\nameref{Star}) and the equality $\acl_0(N)= \dcl_0(N\acl_0(N_0))$.
Remark that \ref{hyp:CoheirLemmaLight} is also a consequence of (\nameref{Star}).
As explained above, the property (\nameref{Star}) holds for theories with geometric elimination of imaginaries:

\begin{fact}\cite[Remarque 1.7]{Blossier_MartinPizarro_2019}
Assume that $T_0$ is stable and has geometric elimination of imaginaries (gEI).
If $a$ is a tuple in $N$ and $C$ a relatively $T_0$-algebraically closed subset of $N$, then the canonical basis $\cb_0(a/C)$ is contained in $\dcl_0(C)$.
In particular, the type $\tp_0(a/C)$ is stationary, so (\nameref{Star}) holds.
\end{fact}

\smallskip

The hypotheses \ref{hyp:InteractionAlgClosure} - \ref{hyp:SameType} will, in the concrete theories we will study in the next sections, be obtained from already known results on these theories.

\medskip

Lastly, in the next sections, the verification of \ref{hyp:ContinuityFSystems} will be simplified by using the following strategy:
\begin{itemize}
	\item[$(*)$] First, we will remark that if $p_\top$ is a minimal completion of a strongly fair system $(p_w, I_w)_{w\in W}$ and $a_\top$ is a realization of $p_\top\cup I_\top$ which belongs to $\mcG$, then for all strict subsets $w_1,..w_m$ of $[n]$, the  subtuple $a_{w_1}...a_{w_m}$ also belongs to $\mcG$.
	
	\item[$(**)$] Second, we will check that if $p_\top$ is a minimal completion of a strongly fair system $(p_w, I_w)_{w\in W}$, then the pair $(p_\top, I_\top)$ is fair, i.e. that the property "being fair" is preserved under amalgamation (after eventually adding some assumptions contained in the fair system properties).
\end{itemize}

\begin{remark}
\label{rk:ContinuityNonforking}
It is useful to remark that step $(*)$ is immediate whenever $\mcG$ is the class of all $\mcL'$-substructure of $N$, since in that case $\mcG$ is closed under taking an $\mcL'$-substructure.

\medskip

Maybe more interestingly, let us show that step $(**)$ is fulfilled in the following situation.
Fix a subset $C\sseq B$ of $N$, and assume that an annotated type $(p(x), I)$  over $B$ is fair if and only if $p$ does not fork over $C$ (in particular, "being fair" does not depend on the annotation $I$, so we shall omit it).
Consider an $n$-amalgamation system $(p_w, w\in W)$ over $B$ such that none of the $p_w$'s fork over $C$ and a completion $p_\top$ of the system. 
We show that show that $p_\top$ does not fork over $C$.

Let $a_\top$ be a realization of $p_\top$: we want to show that $a_\top$ is independent from $B$ over $C$.
We know that for any $w$ in $W$, the independence $a_w\ind_C B$ holds, which yields in particular the independence $(a_i, i\in w) \ind_C B$.
Moreover, the family $(a_i, i\leq n)$ is independent over $B$, so these two facts imply the independence $a_1...a_n \ind_C B$.
Finally, we can remark that $a_\top$ is algebraic over $a_1...a_n$, which concludes the proof.

Note that this proof shows more generally that the property "being nonforking over $C$" is preserved under amalgamation.
\end{remark}

\smallskip

To conclude this subsection, recall that we assumed the simplicity of $T_1$ and that $\ind^1$ is the nonforking independence relation for $T_1$.
However we never used the independence theorem of theory $T_1$, therefore the proofs still hold in a broader context, by simply assuming that $\ind^1$ is a ternary relation defined in $T_1$ with the invariance, monotonicity, base monotonicity, transitivity and closure properties. 

In that case, our results yield in particular that $\ind^1$ satisfies the independence theorem. 
It follows that Theorem \ref{thm:AbstractAmalgamationTransfer} could be used to prove simplicity of the theory $T_1$ even though we will, in the examples below, only apply it to theories already known to be simple. 

\smallskip

\section{Bounded PAC substructures of stable theories} \label{sec:BddPAC}
\subsection{The general result}
\label{sec:PAC}

In this section, we first apply Theorem \ref{thm:AbstractAmalgamationTransfer} to obtain a general result of transfer of $n$-existence for bounded PAC substructures of stable structures (Corollary \ref{cor:nExistencePAC}).
We then particularize this result to specific stable theories, namely that of differentially closed fields of characteristic $0$ and separably closed fields in all characteristics.

Hrushovski proves in \cite{HrushovskiPFFieldsRelStructures} that, given a model $M$ of a strongly minimal theory $T_0$ and a bounded PAC substructure $P$ of $M$ such that $M=\acl_0(P)$, the theory of the pair $(M, P)$ has generalized existence over algebraically closed sets for the independence relation $\ind^0$ (after fixing parameters for an elementary substructure).
Note that generalized existence for the theory of $P$ is not immediately equivalent to this result, but can still be deduced from it.
One must first obtain generalized existence in $P$ for the trace of $\ind^0$ in $P$ (which is a consequence of the fact that $P$ is bounded in $\acl_0(P)$), and then show that the trace of $\ind^0$ coincide with the nonforking independence in $P$ (for example using the Kim-Pillay theorem).

The goal of this subsection is to adapt Hrushovski's proof to the case where $M$ is assumed to be stable but not necessarily strongly minimal. 
To do so, we use the framework developed by Pillay and Polkowska in \cite{Pillay_Polkowska_2006} and by Polkowska in \cite{Polkowska_2007} and we prove that if $P$ is bounded and PAC in $M$ stable, then $P$ has $n$-existence over elementary substructures for all $n$.

Fix a stable theory $T_0$ and a sufficiently saturated model $M$ of $T_0$. 
We start by recalling the definitions and results of \cite{Pillay_Polkowska_2006} and \cite{Polkowska_2007}.

\begin{definition} \label{defBounded} \cite[Definition 2.3 and Proposition 2.5]{Pillay_Polkowska_2006}
Let $P$ be a definably closed substructure of $M$.
We say that $P$ is \textit{bounded in} $M$ if it satisfies one of the following equivalent propositions:
\begin{enumerate}
	\item There exists a cardinal $\lambda$ such that for all $(M', P')$ elementarily equivalent to $(M, P)$, the automorphism group $\aut_0(\acl_0(P')/P')$ has cardinality smaller than $\lambda$.

	\item For all pairs $(M', P')$ elementarily equivalent to $(M, P)$, the automorphism groups $\aut_0(\acl_0(P')/P')$ and $\aut_0(\acl_0(P)/P)$ are isomorphic.

	\item For any elementary extension $(M_0,P_0) \preccurlyeq (M_1,P_1)$ of models of $\text{Th}(M,P)$, we have $\acl_0(P_1)= \dcl_0(\acl(_0P_0) P_1)$.
\end{enumerate}

\end{definition}

The group $\aut_0(\acl_0(P)/P)$ is called the absolute Galois group of $P$ and will be denoted $\text{Gal}(P)$ when there is no risk of confusion regarding the ambient theory.

\smallskip

\begin{definition}  \label{defPAC} \cite[Definition 3.1]{Pillay_Polkowska_2006}
Let $P$ be a substructure of $M$ and $\kappa > |T_0|$ a cardinal. 
We say that $P$ is a $\kappa$\textit{-PAC substructure of} $M$ if every stationary $\mcL_0$-type over a subset $A$ of $P$ of size $|A|<\kappa$ is realized in $P$.
We simply say that $P$ is \textit{PAC} if it is $|T_0|^+$-PAC.
\end{definition}

We recall that this definition is a generalization of the one given by Hrushovski in \cite{HrushovskiPFFieldsRelStructures} in the case of strongly minimal theory.
Given $M$ a strongly minimal structure and a substructure $P$ of $M$, the relation between these notions is as follows: 
if $P$ is $\kappa$-PAC in $M$ in the sense of Definition \ref{defPAC} then $P$ is PAC in $M$ in the sense of Hrushovski.
Conversely, if $P$ is PAC in $M$ in the sense of Hrushovski and $\kappa$-saturated, then $P$ is $\kappa$-PAC in the sense of Definition \ref{defPAC}.

Moreover, Hrushovski showed that a field is perfect and pseudo-algebraically closed (in the algebraic sense) if and only if it is a PAC substructure (in his sense) of its field-theoretic algebraic closure. 
It follows that a field is perfect and pseudo-algebraically closed if and only if it has an elementary extension which is PAC (in the sense of Definition \ref{defPAC}) in an algebraically closed field.

In what follows, we will talk about "pseudo-algebraically closed fields" to refer to the algebraic definition and about "PAC substructures" to refer to Definition \ref{defPAC} in order to avoid any confusion.
In \cite{Hoffmann_Kowalski_2023}, the authors proved that saturated  pseudo-algebraically closed fields correspond to PAC substructures of separably closed fields (see Fact \ref{FirstOrderSCF}).

\begin{definition} \cite[Definition 3.3]{Pillay_Polkowska_2006} \label{def:PACFO}
Let $\mcL_0$ be a language and set $\mcL_P = \mcL_0 \cup \{P\}$, where $P$ is a new unary predicate.
Given a $\mcL_0$-theory $T_0$, we say that the \textit{PAC-property is first order} in $T_0$ (PACFO) if there is a set $\Sigma$ of $\mcL_P$ sentences such that:
\begin{enumerate}
	\item If $P$ is a PAC-substructure of a model $M$ of $T_0$, then $(M, P)\models \Sigma$.
	\item If $(M,P)$ is a $\kappa$-saturated model of $T_0\cup \Sigma$, then $P$ is a $\kappa$-PAC substructure of $M$.
\end{enumerate}
In that case, we will say that a structure (resp. a theory) is $T_0$-PAC if it has an elementary extension (resp. a model) which is PAC in a model of $T_0$.
\end{definition}

\smallskip

In all that follows, we assume that the stable theory $T_0$ has quantifier elimination and PACFO, and that the types in $T_0$ over algebraically closed sets are stationary (TACSS).

Fix a bounded PAC substructure $P$ of a model $M$ of $T_0$ such that  Property (\nameref{Star}) holds for $(M,P)$ and write $T_1=\text{Th}_{\mcL_0}(P)$.
Let $\kappa > |T_0|$ be a cardinal; by PACFO we may assume that both $M$ and $P$ are $\kappa$-saturated.
Assume finally that $\mcL_0$ contains parameters for an elementary substructure of $P$.
We aim to apply the results of Section \ref{sec:GeneralizedAmalgamation} to the theories $T_0$ and $T_1$ (remark that here, the languages $\mcL_0$ and $\mcL_1$ are equal).
The following fact summarize some results obtained in \cite{Polkowska_2007} in this context.

\smallskip

\begin{fact} \label{fact:SummaryPAC}

\begin{enumerate}	
	\item \cite[Lemma 3.8]{Polkowska_2007}
	If $A$ is a relatively $T_0$-algebraically closed subset of $P$, then $\acl_0(A)$ and $\dcl_0(\acl_0(\emptyset)A)$ coincide.

	\item  \cite[Proposition 3.11]{Polkowska_2007} 
	If two subsets of $P$ have the same type in $T_0$, then one is relatively algebraically closed if and only if so is the other. 
	
	\item \cite[Proposition 3.12]{Polkowska_2007}
	If $A$ is relatively algebraically closed in $P$, then we have that $\text{qftp}_1(A) \vdash \tp_1(A)$.
	It follows that if $ B\sseq A$ are relatively algebraically closed, then $\text{qftp}_1(A/B) \vdash \text{tp}_1(A/B)$.
	
	\item \cite[Corollary 3.15]{Polkowska_2007} 
	If $A$ is contained in $P$, then $\acl_1(A)$ equals $\acl_0(A)\cap P$.
\end{enumerate}
\end{fact}

Define the relation $\ind^1$ in $P$ as follows:
for all relatively algebraically closed $A,B$ and $C$ with $C\sseq A\cap B$, we have
\[ A\ind^1_C B \Leftrightarrow A\ind^0_C B.\]
Thus, the independence $\ind^1$ coincides with $\ind^0$ over relatively algebraically closed subsets.
Polkowska showed that $T_1$ is simple and that $\ind^1$ is the nonforking independence in $T_1$ (\cite[Proposition 3.19]{Polkowska_2007}).
Remark that the original article assumes full elimination of imaginaries of $T_0$. However, this is only used to show Property (\nameref{Star}) and TACSS, so the result still holds in our context.

\bigskip
The verifications of \ref{hyp:StrongBdd}-\ref{hyp:SameType} are then easy consequences of the definitions of the classes and Fact \ref{fact:SummaryPAC} (by taking $\mcG$ to be the class of all small subsets of $P$), so it only remains to prove that \ref{hyp:ContinuityFSystems} holds.

Here, a fair type is, by definition, a type in $T_0$ with parameters in $P$ and a realization in $P$, so all the stationary types with parameters in $P$ are fair by the PAC-property.
Moreover Property (\nameref{Star}) yields the converse for types over a relatively algebraically closed subset of $P$.

Remark that $\mcG$ is closed under taking subsets, whence step $(*)$ of Subsection \ref{sec:VerifHyp} holds.
Now if $(p_w, w\in W)$ is a relatively algebraically closed fair amalgamation system in $T_0$ over an existentially closed subset $Q$ of $P$, then all $p_w$'s are stationary.
Note that the Properties (\ref{smallhyp:RelExtension}), (\ref{smallhyp:RelClosure}) and (\ref{smallhyp:StrongBdd}) of Section \ref{sec:RelUniq} are all satisfied by Fact \ref{fact:SummaryPAC}.
It follows by Proposition \ref{prop:RelNUniq} that a minimal completion of $(p_w, w\in W)$ is stationary and hence fair.
This shows that step $(**)$ of Subsection \ref{sec:VerifHyp} holds, therefore \ref{hyp:ContinuityFSystems} is satisfied over existentially closed sets.

\noindent
Finally, we can apply Theorem \ref{thm:AbstractAmalgamationTransfer} to obtain:

\begin{cor} \label{cor:nExistencePAC}
Let $T_0$ be a stable theory with QE, PACFO and TACSS and let $T_1$ be the theory of a bounded PAC substructure of a model of $T_0$. 
Assume that Property (\nameref{Star}) holds.
If $Q$ is an existentially closed subset of a model of $T_1$ containing an elementary substructure such that $T_0$ has $n$-existence over $\acl_0(Q)$, then $T_1$ has $n$-existence over $Q$.
\hfill \qedsymbol
\end{cor}

\medskip

To conclude this subsection, we observe that if the theory $T_0$ is $\omega$-stable, then Corollary \ref{cor:nExistencePAC} can be somewhat simplified.
In this situation, we get that every type has a well-defined Morley rank, as well as finite multiplicity (which is equal to the Morley degree).

\begin{fact} \label{ElemSubsttr}
\cite[Corollary 3.10]{Pillay_Polkowska_2006} Let $T_0$ be stable with elimination of imaginaries such that every finitary type has finite multiplicity.
The algebraic closure of any $\kappa$-PAC substructure of a model $M$ of $T_0$ is an elementary substructure of $M$ (and it is $\kappa$-saturated).
\end{fact}

Moreover, it was remarked in \cite[Proposition 4.2]{Hoffmann_Kowalski_2023} that if $T_0$ has the definable multiplicity property, then we can express whether a type is stationary, which implies that the PAC property is first-order.

\begin{cor}
Let $T_0$ be an $\omega$-stable theory with QE, EI and the definable multiplicity property.
If $T_1$ is bounded and PAC in $T_0$, then $T_1$ has $n$-existence over elementary substructures.
\hfill \qedsymbol
\end{cor}

\smallskip

As an example of $\omega$-stable theory we consider the theory $DCF_{0,m}$ of differentially closed fields of characteristic $0$ with $m$ commutating derivations in the language of rings with additional unary function symbols for the derivations.
A detailed presentation of the theory $DCF_{0,m}$ can be found in \cite{marker} and \cite{Marker2000ModelTO}. 

For $m\geq 1$, the theory $DCF_{0,m}$ is $\omega$-stable and eliminates quantifiers and imaginaries.
Moreover it has PACFO, even though it does not have the definable multiplicity property (see \cite{Pillay_Polkowska_2006, Sanchez_Tressl_2020_DiffLargeFields}).
We will see in in Subsection  \ref{sec:FieldsOperators} that $DCF_{0,m}$ has $n$-existence over any algebraically closed set, so we can apply Corollary \ref{cor:nExistencePAC} to deduce the following:

\begin{cor} \label{cor:DCF0PAC}
All bounded $DCF_{0,m}$-PAC theories have $n$-existence over existentially closed substructures containing an elementary substructure for all $n$.
\hfill \qedsymbol
\end{cor}

\subsection{Separably closed fields} \label{sec:SCFPAC}
As explained above, Hrushovski showed in \cite{HrushovskiPFFieldsRelStructures} that perfect bounded pseudo-algebraically closed fields have generalized existence over algebraically closed sets containing elementary substructure.
In this section, we aim to generalize this result to all bounded pseudo-algebraically closed fields.

To do so, we apply Corollary \ref{cor:nExistencePAC} to bounded PAC substructures of separably closed fields, as Hoffmann and Kowalski proved in \cite{Hoffmann_Kowalski_2023} that such structures correspond exactly to bounded pseudo-algebraically closed fields (see Fact \ref{FirstOrderSCF}).
Let $\mcL$ be the language of rings.

\medskip

Let $K$ be a field of characteristic $p>0$.
We say that a field extension $L/K$ is \textit{separable} if $K$ is linearly disjoint from the field $L^p=\{l^p; l\in L\}$ over the field $K^p$.
Moreover, the field $K$ is said to be \textit{separably closed} if it has no nontrivial algebraic separable extension.
A \textit{separable closure} of $K$ is a separably closed algebraic extension of $K$. 
Such an extension always exists and is unique up to $K$-isomorphism; it is denoted $K^s$.

If $K$ is a field of characteristic $p$, then $K$ is in particular a $K^p$-vector space, and the dimension $[K:K^p]$ is either infinite or of the form $p^e$, with $e$ an integer.
The imperfection degree $e(K)$ of $K$ is then equal to $\infty$ (if it is the former) or to $e$ (if it is the latter).
The $\mcL$-theory of separably closed fields of characteristic $p$ and imperfection degree $e$ is complete (see \cite{Ershov_1967}); it is denoted by $SCF_{p,e}$.

\medskip

Separably closed fields can be studied in the language $\mcL$ of rings as explained above, but it is not the only possibility.
One can also consider them in the language $\mcL(\lambda)$ obtained by adding several function symbols $\lambda_{n,\alpha}$ ($\alpha\in [0,p-1]^n$) to the language $\mcL$, as defined in \cite{Delon_1988}.
A $n$-tuple $b$ is said to be $p$-\textit{free} is the family of its $p$-monomials is $K^p$-independent.
In that case, if $a$ is an element of the $K^p$-vector space generated by the $p$-monomials of $b$, then the $\lambda_{n,\alpha}$ are uniquely defined by
$$a= \sum_{\alpha \in [0,p-1]^n}\lambda_{n,\alpha}(a,b)^p b^\alpha.$$
Otherwise, set $\lambda_{n,\alpha}(a,b)=0$.
This new theory will be denoted $SCF_{p,e}(\lambda)$; it is also complete and has quantifier elimination but does not eliminate imaginaries if $e$ is nonzero (see \cite[Proposition 43]{Delon_1988}).

Finally, if the imperfection degree $e$ is finite, we can also consider the expansion $SCF_{p,e}(b, \lambda)$ of $SCF_{p,e}$ to the language $\mcL(b,\lambda)$, as defined in \cite{Delon_1988}. 
In that case, the tuple $b$ is a finite tuple of constants interpreted as a $p$-basis and the $\lambda$'s are unary functions defined as the partial evaluations $\lambda_{e, \alpha}(\cdot, b)$ of the $\lambda$-function defined above.
This theory is complete and eliminates quantifiers and imaginaries.

\medskip

Our goal is to apply Corollary \ref{cor:nExistencePAC} to bounded PAC substructures in both theories $SCF_{p, \infty}(\lambda)$ and $SCF_{p,e}(b, \lambda)$, $e$ finite. 
Recall that the theory $SCF_{p,e}(b,\lambda)$ has elimination of imaginaries but $SCF_{p, \infty}(\lambda)$ does not. Bartnick proved in \cite{bartnick_2024} that types over algebraically closed sets in $SCF_{p, \infty}(\lambda)$ are stationary, so we will only have to show that the property (\nameref{Star}) holds for $SCF_{p, \infty}$-substructures.

\begin{fact} \label{SeparableExtension} Some results about separable extensions and separably closed fields:
\begin{enumerate}
	\item \cite[Fact 1]{Srour_IndependenceSCF}
	If $L/K$ is a separable algebraic extension, then $L$ and $K$ have the same imperfection degree: $e(L)=e(K)$.
	
	\item \cite[Paragraph 1.4]{Chatzidakis_2004}
	A field extension $L/K$ is separable if and only if any $p$-basis of $K$ is still $p$-independent in $L$, or, equivalently, if $K$ is closed under the $\lambda$-functions of $L$.
	
	\item \cite[Paragraph 1.6]{Chatzidakis_2004} Let $M\models SCF_{p,e}$ be a separably closed field and $A$ a subset of $M$. The model-theoretic definable closure $\dcl(A)$ is the field generated by the closure of $A$ under the $\lambda$-functions of $M$, and its model-theoretic algebraic closure is the separable closure of $\dcl(M)$.
\end{enumerate}
\end{fact}

To apply Corollary \ref{cor:nExistencePAC} to PAC substructures of a theory $T_0$, we first need to verify that the initial theory $T_0$ has $n$-existence.
Here, in the case $e$ finite, the theory $SCF_{p,e}(b, \lambda)$ is a stable theory with quantifier elimination, so it follows by Fact \ref{Sharp} that $SCF_{p,e}(b, \lambda)$ has $n$-existence over elementary substructures for all $n$.
In the infinite case, the theory $SCF_{p,\infty}(\lambda)$ is still stable but does not eliminate imaginaries.
We however know that types over algebraically closed sets are stationary, so we get that $SCF_{p,\infty}(\lambda)$ has $n$-amalgamation over elementary substructures by Corollary \ref{cor:nUniqu}.

\smallskip

Let us now check that the algebraic closure of a PAC substructure is an elementary substructure.
Assume that $T_0$ is either $SCF_{p,e}$, $SCF_{p,e}(\lambda)$ or $SCF_{p,e}(b, \lambda)$. 
Let $A$ be an algebraically closed subset of a model $M$ of $T_0$ containing a $p$-free family of $M$ of size $e$ (which is always the case if $T_0$ is either $SCF_{p,e}(b, \lambda)$, $e$ finite).
By Fact \ref{SeparableExtension} and quantifier elimination, we then get that $A$ is an elementary substructure of $M$.

Note that this yields an alternative proof of $n$-existence in $SCF_{p,\infty}(\lambda)$:  in $SCF_{p,\infty}(\lambda)$, any algebraically closed subset of a model of $SCF_{p, \infty}$ containing an elementary substructure is again an elementary substructure, so $n$-existence follows by Lemma \ref{CompletionLemma}.

\begin{lemma}
We assume here that $T_0$ is either $SCF_{p,e}(b, \lambda)$ ($e$ finite) or $SCF_{p, \infty}(\lambda)$.
If $P$ is a bounded PAC substructure of a model $M$ of $T_0$, then $\acl_M(P)$ is an elementary substructure of $M$.
\end{lemma}

\begin{proof}
If $T_0$ is $SCF_{p,e}(b, \lambda)$, then this is immediate by the above discussion, so we may assume that $T_0$ is $SCF_{p, \infty}(\lambda)$.
It was proved in \cite{Delon_1988} that the the type of an element which does not belong to the $K^p$-vector space generated by the parameters is stationary, so the PAC-property implies that $P$ (and hence $\acl_M(P)$) contains an infinite $p$-independent family. 
It follows that $\acl_M(P)$ is an elementary substructure of $M$.
\end{proof}

In \cite{Hoffmann_Kowalski_2023},  the authors give the following  description of PAC substructures of separably closed fields (see \cite[Theorem 4.18]{Hoffmann_Kowalski_2023} and the following discussion for the first point, and \cite[Theorem 4.23]{Hoffmann_Kowalski_2023} and the following discussion for the second point).
Note that our definition of a PAC substructure and the one from \cite{Hoffmann_Kowalski_2023} are in general different, but coincide when the considered structures are sufficiently saturated.
In particular, the results obtained in \cite{Hoffmann_Kowalski_2023} still hold in our context.

\begin{fact}  \label{FirstOrderSCF}
\
\begin{enumerate}
	\item Let $e$ be finite and  $K$ a substructure of some large model $M \models SCF_{p,e}(b, \lambda)$. 
	The structure $K$ is $SCF_{p,e}(b, \lambda)$-PAC  if and only if $K$ is a pseudo-algebraically closed field.
	In particular, the PAC-property is first order for $SCF_{p,e}(b, \lambda)$.
	
	\item Let $K$ be a substructure of some large model $M \models SCF_{p, \infty}(\lambda)$. 
	The structure $K$ is $SCF_{p, \infty}(\lambda)$-PAC if and only if $K$ is a pseudo-algebraically closed field with infinite imperfection degree.
	In particular, the PAC-property is first order for $SCF_{p, \infty}(\lambda)$.
\end{enumerate}
\end{fact}

Remark that in the case of the characterization of $SCF_{p,e}(b, \lambda)$-PAC structures, it is assumed that $K$ is a substructure in the language $\mcL_{b,\lambda}$, so we automatically have $e(K)=e$.
If $K$ is simply a substructure in the language $\mcL$, then we get that $K$ is PAC in $M$ if and only if $K$ is a pseudo-algebraically closed field of imperfection degree $e$. 

Note also that a PAC-substructure of a field is always a field, as PAC-substructure are definably closed.

\bigskip

We know that $SCF_{p,e}(b,\lambda)$ with $e$ finite is stable, eliminates quantifiers as well as imaginaries and has PACFO by \ref{FirstOrderSCF}. 
We can therefore apply Corollary \ref{cor:nExistencePAC}, which yields:

\begin{cor} \label{SCFfinite}
Any theory which is bounded PAC in  $SCF_{p,e}(b, \lambda)$ ($e$ finite) has $n$-existence over over existentially closed substructures containing an elementary substructure for any integer $n$.
\hfill \qedsymbol
\end{cor}

It only remains to see that if $P$ is a bounded PAC substructure of a model $M$ of $T=SCF_{p,e}(\lambda)$, then Property (\nameref{Star}) holds, as Theorem \ref{thm:AbstractAmalgamationTransfer} will then imply real $n$-amalgamation for $P$ over existentially closed sets containing elementary substructures.
To do so, we will use the following fact:

\begin{fact} \cite{Srour_IndependenceSCF}
Let $M$ be a separably closed field with infinite imperfection degree, and let $A$ and $B$ be two subfields of $M$ such that the extensions $M/A$ and $M/B$ are separable.
\begin{enumerate}
	\item If $A$ and $B$ have the same algebraic type, then they have the same $\mcL(\lambda)$-type in $M$.
	
	\item Let $C$ be a common subfield of $ A$ and $B$ such that the extension $M/C$ is separable. We have
\[ 
	A \ind_{C}^{\lambda} B 
	\Leftrightarrow 
	A \ind_{C}^{ACF} B, \,\, A \ind_{CA^p}^{ld} BA^p \text{ and } M/\langle AB\rangle_{ACF} \text{ is separable.} 
\]
where $\ind^\lambda$ is the nonforking independence in  $SCF_{p, \infty}(\lambda)$ and $\ind^{ACF}$ the nonforking independence in  $ACF_p$, while $\langle AB\rangle_{ACF}$ is the (field theoretic) composite of the fields $A$ and $B$.
	
\end{enumerate}
\end{fact}

\begin{lemma}
Let $P$ be a bounded PAC substructure of a model $M$ of the theory $T=SCF_{p, \infty}(\lambda)$. 
 Property (\nameref{Star}) holds for the pair $(M, P)$.
\end{lemma}

\begin{proof}
Let $C$ be a subset of $P$   which is definably closed and relatively algebraically closed in $P$.
Let $a$ in $P$ and $a'$ in $M$ be such that $\tp_{\lambda}(a/C)=\tp_{\lambda}(a'/C)$, and both independence $a\ind_C^{\lambda} B$ and $a'\ind_C^{\lambda} B$ hold, where $B$ is an $\mcL(\lambda)$-definably closed subset of $M$ containing $C$.
Let us show that $A=\dcl_{\lambda}(Ca)$ and $A'=\dcl_{\lambda}(Ca')$ have the same $\mcL(\lambda)$-type over $B$, as it will prove that $\tp_{\lambda}(a/C)$ is stationary.

\smallskip

We know that the field-theoretic canonical basis $\cb_{ACF}(A/C)$ of $A$ over $C$ in $\mcL$ is contained in both $\acl_{ACF}(C)$ and $\dcl_{ACF}(CA)=\dcl_{ACF}(A)$.
Moreover, the field-theoretic definable closure $\dcl_{ACF}(A)$ of $A$ is the perfect hull of $A$
Hence every element of $\dcl_{ACF}(A)$ is (field-theoretically) interdefinable with an element of $A$.
In particular, we may assume that $\cb_{ACF}(A/C)$ is contained in $A \cap \acl_{ACF}(C)$.

\noindent
Remark also that $A$ is contained in $P$, so we have
\begin{align*}
	A\cap\acl_{ACF}(C) & = A\cap \acl_{ACF}(C) \cap M  \\
				 & \sseq A\cap \acl_{\lambda}(C) \\
				 & = A\cap \acl_{\lambda} (C) \cap P\\ 
				 & \sseq A \cap C = C,
\end{align*}
i.e. the canonical basis $\cb_{ACF}(A/C)$ is contained in $C$, whence  $\tp_{ACF}(A/C)$ is stationary.
Now the independence $A\ind_C^{\lambda} B$ yields $A\ind_C^{ACF} B$, and the same holds for $A'$, so $A$ and $A'$ have the same algebraic type over $B$.
In particular, the composita $\langle AB\rangle_{ACF}$ and $\langle A'B\rangle_{ACF}$ have the same algebraic type.

\smallskip
\noindent
By the characterization of $\ind^\lambda$, we get that both $M/\langle AB\rangle_{ACF}$ and $M/\langle A'B\rangle_{ACF}$ are separable extensions, whence $\langle AB\rangle_{ACF}$ and $\langle A'B\rangle_{ACF}$ have the same $\mcL(\lambda)$-type.
It follows that $A$ and $A'$ have the same $\mcL(\lambda)$-type over $B$ in $M$, which ends the proof.
\end{proof}

\smallskip
\noindent
Finally, we can apply Corollary \ref{cor:nExistencePAC} to obtain the following result:

\begin{cor} \label{cor:namalgSCFinfiniPAC}
Any $\mcL(\lambda)$-structure which is bounded PAC in the theory $SCF_{p, \infty}$ has $n$-existence over existentially closed substucture containing an elementary substructure, for any integer $n$.
\hfill \qedsymbol
\end{cor}

We conclude by remarking that Corollaries \ref{SCFfinite} and \ref{cor:namalgSCFinfiniPAC} yield $n$-existence for all bounded pseudo-algebraically closed fields $k$.
Recall that if $k$ is perfect, then we already know that the theory of $k$ in the language of rings has $n$-existence for all $n$ by \cite{HrushovskiPFFieldsRelStructures}. 
We may assume that $k$ is $\aleph_0$-saturated by replacing it with an elementary extension.
Write $e$ its imperfection degree: the separable closure $k^s$ of $k$ is a separably closed field of imperfection degree $e$, i.e. $k^s\models SCF_{p,e}$, and the extension $k^s/k$ is separable.
By Fact \ref{SeparableExtension}, this shows that $k$ is a substructure of $k^s$ in the language of rings with $\lambda$-functions.

In the case where the imperfection degree $e$ is infinite, this implies that the expansion $(k, \lambda)$ of $k$ by the $\lambda$-functions of $k^s$ is a PAC-substructure of the theory $SCF_{p, \infty}$ by Fact \ref{FirstOrderSCF}.
In the case where $e$ is finite, the separability of $k^s/k$ means that we can choose a $p$-basis $b$ of $k^s$ in $k$, and then $(k,b, \lambda)$ is a PAC-substructure of $SCF_{p,e}$ by Fact \ref{FirstOrderSCF} again.

In both cases, we have found an expansion of $k$ by definable sets whose theory admits $n$-existence over existentially closed substructure containing an elementary substructure, therefore the same holds for the theory of $k$.

\begin{cor}
Bounded pseudo-algebraically closed fields have $n$-existence over existentially closed substructure containing an elementary substructure for all $n$.
\hfill \qedsymbol
\end{cor}

\section{Fields with operators} \label{sec:FieldsOperators}

The goal of this section is to apply Theorem \ref{thm:AbstractAmalgamationTransfer} to various theories of algebraically closed fields with operators.
The proof we present here is a direct generalization of the proof of $n$-existence in \cite{Chatzidakis_1999}.
We fix in this section a complete stable theory $T_0$  with elimination of quantifiers and such that types over algebraically closed sets are stationary.
In particular, the theory $T_0$ has $n$-uniqueness over existentially closed substructures for all $n$ by Corollary \ref{cor:nUniqu}.

Define a new language $\mcL_1=\mcL_0\cup \mcL'$, where $\mcL'$ is a language containing only new unary function symbols.
Assume that $T_1$ is a complete and model-complete simple expansion of $T_0$ to $\mcL_1$ such that for every symbol $h$ in $\mcL'$ and every $\mcL_0$ term $t(x_1,...,x_n)$, there are $\mcL'$-terms 
$t_1^{h,t}(x_{i_1}),...,t_m^{h,t}(x_{i_m})$ and an $\mcL_0$-term $s^{h,t}(y_1,...,y_m)$ such that $T_1$ implies
\[
	\forall x_1,...,x_n, h(t(x_1,..., x_n)) = s^{h,t}(t_1^{h,t}(x_{i_1}),...,t_m^{h,t}(x_{i_m}))
\]

\begin{example}
If $a$ and $b$ are elements of a differentially closed field $(K, \delta)$, then we have $\delta(a+b)=\delta(a)+\delta(b)$ and $\delta(ab)= a\delta(b) +\delta(a)b$.
It follows that for every polynomial $P(x_1,...x_n)$, there is a polynomial $Q(x_1,...x_n,y_1,...y_n)$ such that
$$
	\forall x_1,...,x_n, \delta(P(x_1,...x_n) = Q(x_1,...x_n, \delta(x_1),...,\delta(x_n)).
$$

\end{example}

Fix a $\kappa$-saturated model $N$ of $T_1$ for $\kappa > |T_1|$; we denote by $M$ the restriction of $N$ to $\mcL_0$. 
Note that $M$ is a $\kappa$-saturated model of $T_0$.
We make the following assumptions on small subsets of $N$:
\begin{enumerate}
	\item \label{operators:AlgClos} If $A$ is a $\mcL'$-substructure of $N$, then $\acl_1(A)$ coincide with $\acl_0(A)$.
	\item \label{operators:CaracIndep} If $A,B$ and $C$ are $\mcL'$-substructures on $N$ such that $C\sseq A\cap B$, then we have 
	\[
		A\ind_C^1 B \Leftrightarrow A\ind_C^0 B,
	\]
	where $\ind^1$ is the nonforking independence in the simple theory $T_1$ and $\ind^0$  the nonforking independence in the stable theory $T_0$.
	\item \label{operators:SameType} If $a$ and $b$ are algebraically closed in $T_1$ and have the same quantifier-free $\mcL_1$-type, then they have the same type in $T_1$.
	\item \label{operators:ModelCompletion}
		The theory $T_1$ is the (unique) model-completion of its universal part $T_1^\forall$.
		Recall that $T_1^\forall$ is the common theory of all substructures of $N$. 
	
	\item \label{operators:ExtensionAmalgamation} Let $C \sseq A$ be substructures of $N$ such that $A\cap B = C$, where $B=\acl_0(C)$.
	If $(B,H)$ is an $\mcL'$-structure on $B$ (not necessarily the one induced by $N$) such that $(B, H)$ is a model of $T_1^\forall$ and an extension of $C$ (as an $\mcL_1$-structure), then there exists an $\mcL'$-structure  $(\langle AB\rangle_0, H')$ on the generated $\mcL_0$-substructure $\langle AB\rangle_0$ extending both $A$ and $(B,H)$ and such that $(\langle AB \rangle_0, H')$ is a model of $T_1^\forall$.
\end{enumerate}

\begin{remark}
\begin{itemize}
	\item If a tuple $a$ of $N$ is closed under $\mcL'$, then so is $\langle a \rangle_0$. 
	In particular, the generated substructures $\langle a \rangle_0$ and $\langle a \rangle_1$ coincide.
	
	\item If $a$ and $b$ are $\mcL'$-closed tuples of $N$ such that $\langle a \rangle_0$ and $\langle b \rangle_0$ have the same quantifier-free $\mcL_0$-type and $a$ and $b$ have the same quantifier-free $\mcL'$-type, then $\langle a \rangle_0$ and $\langle b \rangle_0$ have the same quantifier-free $\mcL_1$-type.
	
	\item Assume that $T_0$ is a theory of separably closed fields with the corresponding $\lambda$-functions and the multiplicative inverse.
	If $C\sseq A$ are substructures of $N$, then they are closed under the $\lambda$-functions of $M$. 
	In particular the extension $M/C$ is separable, whence so is the extension $A/C$.
	If moreover $A\cap C^{sep}$ equals  $C$, then the extension $A/C$  is regular and hence $A$ and $B$ are linearly disjoint over $C$.

It follows that proving (\ref{operators:ExtensionAmalgamation}) reduces to showing that if $A$ and $B$ are models of $T_1^\forall$ which are linearly disjoint over $C$, then the $\mcL'$-structure on $A\cup B$ extends to the composite field $AB$ in such a way that $AB$ is a model of $T_1^\forall$ (which then embeds into $N$ by (\ref{operators:ModelCompletion})): this is satisfied in all the examples below.
\end{itemize}
\end{remark}

\begin{example} \label{ex:ListOfFieldsWithOperators}
This context contains several already known and well-studied theories of fields.
In the case where $T_0$ is the theory of algebraically closed fields of characteristic $p$ in the language $\mcL_0$ of rings with the multilplicative inverse,  the following theories meet the required hypotheses: 
\begin{itemize}[-]
	\item In characteristic zero, the theory $DCF_{0,m}$ of differentially closed fields with $m$ commuting derivations.
	
	\item Any completion of the theory $ACF_pA$ of algebraically closed fields in characteristic $p$ with a generic automorphism \cite{ChatzidakisHrushovski_1999}, in the language $\mcL_0$ augmented with symbols for the automorphism and its inverse.

	\item In characteristic zero, any completion of the theory $DCF_0A$ of differentially closed fields with a generic automorphism commuting with the derivation \cite{Medina_2007}, after fixing parameters for the algebraic closure of the empty set.
	
	\item In characteristic zero, completions of theories of algebraically closed fields with free operators $\mathcal{D}$-$CF_0$, as described in \cite{MoosaScanlon_2014}, after fixing parameters for the algebraic closure of the prime field.
\end{itemize}

There are also examples in wich the theory $T_1$ is an expansion of the theory $T_0=SCF_{p, e}$ of (imperfect) separably closed fields of imperfection degree $e$, in the language $\mcL_0$ of rings (possibly with $\lambda$-functions, and with a $p$-basis if $e$ is finite).
Among others:
\begin{itemize}[-]
	\item The theory $DCF_p$ of differentially closed fields of characteristic $p$ (if the imperfection degree is infinite) \cite{Wood_1976}.
	
	\item Any completion of the theory $SCF_{p,e}A$ of separably closed fields with a generic automorphism \cite{Chatzidakis_2001}, after fixing parameters for the algebraic closure of the empty set.
\end{itemize}

Note that the theories $ACF_pA$, $DCF_0A$, $SCF_{p,e}A$ and $\mathcal{D}$-$CF_0$ are model-complete but not complete. 
In particular, if $T_1$ is some completion of one of these theories, then we add a set of parameters to ensure that $T_1$ is the model completion of the common theory of its substructures.
\end{example}

\medskip

Define $\mcG$ to be the class of all small $\mcL'$-substructures of $N$, and let us check that hypotheses \ref{hyp:StrongBdd}-\ref{hyp:ContinuityFSystems} are satisfied. Note that \ref{hyp:StrongBdd} and \ref{hyp:CoheirLemmaLight} are trivial, as the reducts to $\mcL_0$ models of $T_1$ are models of $T_0$.
We also see that \ref{hyp:InteractionAlgClosure} is  a consequence of the definition of $\mcG$ and the assumption (\ref{operators:AlgClos}).
Moreover, for all $\mcL'$-closed sets $a$ and $b$, their union $ab$ is also $\mcL'$-closed since $\mcL'$ contains only unary function symbols.
It then follows from the assumption (\ref{operators:CaracIndep}) that \ref{hyp:CaracIndep} holds.

To see that \ref{hyp:SameType} holds, remark that if $a$ and $b$ have the same type in $T_0$ and $b$ is algebraically closed in $T_1$, then $b$ is in particular algebraically closed in $T_0$ and therefore so is $a$.
If moreover $a$ is $\mcL'$-closed and has the same quantifier-free $\mcL'$-type as $b$, then $a$ is algebraically closed in $T_1$ by the assumption (\ref{operators:AlgClos}), and has the same quantifier-free $\mcL_1$-type as $b$.
It follows by (\ref{operators:SameType}) that $a$ and $b$ have the same type in $T_1$, which proves \ref{hyp:SameType}.

Note that by definition of $\mcG$, fair types are exactly annotated types $(p(x),I(x))$ such that $p(x)\cup I(x)$ is consistent. 
It follows that every fair system is strongly fair.
In order to prove \ref{hyp:ContinuityFSystems}, it only remains to show that any minimal completion of an algebraically closed strongly fair system is a fair type:

\begin{lemma}
Let $(p_w(x_w), I_w)_{ w\in W}$ be an algebraically closed strongly fair system over a set $Q$ which is $\mcL_0$-existentially closed in $M$.
For any minimal completion $p_\top(x_\top)$ of $(p_w(x_w), w\in W)$, the pair $(p_\top, I_\top)$ is fair.
\end{lemma}

\begin{proof}
Let $a_\top$ be a realization of $p_\top(x_\top)$.
Keeping the notations of the proof of Proposition \ref{prop:RelNUniq}, we set $\Delta_k= \{w\in W| [k,n] \sseq w \text{ and } |w|=n-1\}$ and $\Gamma_k= \{ w\in W| [k+1,n] \sseq w, k\notin w \text{ and } |w|=n-2 \}$.
Define $p_{\Delta_k}= \tp(a_{\Delta_k}/Q)$; let us show by induction on $k$ that $p_{\Delta_k}\cup I_{\Delta_k}$ has a realization $b_{\Delta_k}$ belonging to $\mcG$.

\smallskip

The case $k=3$ is a simple application of extension in the simple theory $T_1$, stationarity over $Q$ in $T_0$ and assumption (\ref{operators:SameType}).
Let now $b_\top$ be a realization of $p_\top$ such that $b_{\Delta_k}$ realizes $I_{\Delta_k}$ for $k\geq 3$. 
Let  $c_{[n]\setminus \{k\}}$ be a realization of $p_{[n]\setminus \{k\}}\cup I_{[n]\setminus \{k\}}$. 
Then both $\tp_0(b_{\Gamma_k}/Q)$ and $\tp_0(c_{\Gamma_k}/Q)$ are minimal completions of the $(n-1)$-amalgamation system $(p_v, v\subsetneq [n]\setminus \{k\})$, so they are equal since Corollary \ref{cor:nUniqu} yields $(n-1)$-uniqueness over $Q$.
We also know that $b_{\Gamma_k}$ and $c_{\Gamma_k}$ have the same quantifier-free $\mcL'$-type by definition of a strongly fair system. 
It follows that $\langle b_{\Gamma_k} \rangle_1$ and $\langle c_{\Gamma_k}\rangle_1$ have the same quantifier-free $\mcL_1$-type over $Q$; let $f: \langle c_{\Gamma_k} \rangle_1 \to \langle b_{\Gamma_k} \rangle_1$ be the associated $ \mcL_1$-isomorphism.

Note that $f$ extends to an $\mcL_0$-isomorphism $\tilde{f}$ from $c_{[n]\setminus \{k\}}$ to $b_{[n]\setminus \{k\}}=\acl_0(\langle b_{\Gamma_k} \rangle_1)$.
Consider the $\mcL_1$-structure $b_{[n]\setminus \{k\}}'$ on $b_{[n]\setminus \{k\}}$ obtained by by transfering the $\mcL_1$-structure on $c_{[n]\setminus \{k\}}$ via $\tilde{f}$.
Note that $b_{[n]\setminus \{k\}}'$ is not necessarily a substructure of $N$ but is a model of $T_1^\forall$.
The structures $\langle b_{\Delta_k} \rangle_0$ and $\langle b_{\Gamma_k} \rangle_0$ are also models of $T_1^\forall$, and they are both definably closed in $T_0$.
Moreover, Lemma \ref{lemma:SharpEntailment} yields that $\langle b_{\Gamma_k} \rangle_0$ is relatively algebraicaly closed in $\langle b_{\Delta_k} \rangle_0$.
It follows by assumption (\ref{operators:ExtensionAmalgamation}) that $b_{[n]\setminus \{k\}}'\langle b_{\Delta_k} \rangle_0$ embeds into a model $D$ of $T_1^\forall$, which in turns embeds into $N$ since $T_1$ is the unique model-completion of $T_1^\forall$.

Let $j:D \to N$ be this embedding.
Since $Q$ is algebraically closed, we may assume by (\ref{operators:SameType}) that it is fixed by $j$.
The map $j$ is an $\mcL_0$-isomorphism, therefore $j(b_{[n]\setminus \{k\}}'b_{\Delta_k})$ has the same $\mcL_0$-type as $b_{[n]\setminus \{k\}}'b_{\Delta_k}=b_{\Delta_{k+1}}$ by quantifier elimination.
Moreover $j(b_{[n]\setminus \{k\}}')$ and $j(b_{\Delta_k})$ have the same quantifier-free $\mcL'$-type as $b_{[n]\setminus \{k\}}'$ and $b_{\Delta_k}$ respectively; it follows that $j(b_{[n]\setminus \{k\}}'b_{\Delta_k})$ realizes $p_{\Delta_{k+1}}I_{\Delta_{k+1}}$.
\end{proof}

\begin{cor}
Let $T_0$ be a stable theory which eliminates quantifiers and such that types over algebraically closed sets are stationary.
Assume that $T_1$ is as described in this subsection and satisfies the assumptions (\ref{operators:AlgClos})-(\ref{operators:ExtensionAmalgamation}). 
The theory $T_1$ has $n$-existence for all integer $n$ over algebraically closed substructure which are existentially closed in $T_0$.
\end{cor}

In particular, all the theories of example \ref{ex:ListOfFieldsWithOperators} have $n$-existence over algebraically closed substructure, for all $n$.

\smallskip

\section{Pairs of structures} 
\label{sec:PairsOfStructures}

\subsection{Structures with a notion of $P$-independence}  \label{sec:StructurePIndependence}
In this subsection, we use Theorem \ref{thm:AbstractAmalgamationTransfer} to obtain a general result of transfer of $n$-existence for pairs of structures (Corollary \ref{cor:nExistPairs}).
We then apply this result to well-known theories of pairs, namely that of lovely pairs, $H$-structures and bounded PAC beautiful pairs.

We assume here that $T_0$ is a simple theory with quantifier elimination in a language $\mcL_0$, and that $T_1$ is a simple expansion of $T_0$ to the language $\mcL_1$ obtained by adding a single unary predicate $P$ to $\mcL_0$.

We fix a $\kappa$-saturated model $N=(M,P_M)$ of $T_1$.
For any subset $A$ of $N$, we denote by $P_A$ its intersection with $P_M$.

\begin{definition}  \label{def:PIndependentSet}
A subset $A$ is said to be $P$\textit{-independent} if $A$ is $\ind^0$-independent from $P_M$ over $P_A$.
Take $\mcG$ to be the class of all small $P$-independent subsets.
	Remark that if an element $A$ of $\mcG$ is algebraically closed in $T_0$, then $P_A$ is relatively algebraically closed in $P_M$.
\end{definition}

Since the language $\mcL'$ is composed of a single unary predicate, an annotated type is simply a pair $(p(x), \opart{x})$, where $\opart{x}$ is a subtuple of $x$; it is fair if $p$ has a $P$-independent realization $a$ such that $\opart{a}:= a\rest \opart x$ coincides with $P_a$.

\begin{notation}
Given an annotated type $(p(x), \opart x)$ over $C$, denote $\opart{p}(\opart{x})$ its restriction to the variables $\opart x$, and $\widehat{p}(\opart x)$ the restriction of $\opart{p}(\opart{x})$ to the parameters $P_C$.
In other words, if $a$ realizes $p$, then $\opart{p}(\opart{x})=\tp_0(\opart a/ C)$ and $\widehat{p}(\opart x) = \tp_0(\opart a/ P_C)$.

Given a subset $D$ of $P_M$, let $\mcC(D)$ denote the set of $\mcL_0$-types over $D$ with a realization in $P$.
We say that $\mcC(D)$ is closed under relatively algebraically closed amalgamation if whenever $q_\top$ is a minimal completion to an amalgamation system which is algebraically closed relatively to $P$, then $q_\top$ has a realization in $P$.
\end{notation}

\noindent
We additionally make the following assumptions:
\begin{enumerate}
	\item \label{pairs:AlgClosure}
	A set which is algebraically closed in $T_1$ is $P$-independent, and the algebraic closure of a $P$-independent set in $T_1$ coincide with its algebraic closure in $T_0$.
	
	 \item \label{pairs:CaracIndep}
	 For every $P$-independent $a, b$ and $c$ with $c\sseq a\cap b$, we have
	 \[
	 	a\ind^1_c b \Leftrightarrow a \ind_c^0 b \text{ and } ab \text{ is } P\text{-independent.}
	 \]
	 
	 \item \label{pairs:RelAlgClos}
	 For all tuples $a$ and $b$ in $P_M$ with the same type in $T_0$, if $b$ is relatively algebraically closed in $P_M$, then so is $a$.
	 
	 \item \label{pairs:SameType}
	 If $a$ and $b$ are algebraically closed in $T_1$ and with the same quantifier-free $\mcL_1$-type, then they have the same type in $T_1$.

	 \item \label{pairs:PIndependentRealizations}
	 Given an algebraically closed $P$-independent set $Q$ and an annotated type $(p(x), \opart{x})$, if $\widehat{p}(\opart{x})$ has a realization in $P$ and $\opart{p}(\opart{x})$ does not fork over $P_Q$, then $p(x)$ has a realization $a$ which is independent from $P$ over $\opart a$ and such $\opart a$ is contained in $P$. 
\end{enumerate}

We now check that these assumptions are enough for the hypotheses of Theorem \ref{thm:AbstractAmalgamationTransfer} to hold.
Note that \ref{hyp:StrongBdd} and \ref{hyp:CoheirLemmaLight} are immediate since the reducts to $\mcL_0$ of models of $T_1$ are models of $T_0$.
Moreover, the hypotheses  \ref{hyp:InteractionAlgClosure} and \ref{hyp:CaracIndep} are is immediately equivalent to the assumptions (\ref{pairs:AlgClosure}) and (\ref{pairs:CaracIndep}) respectively.
Remark also that the assumptions (\ref{pairs:AlgClosure}) and (\ref{pairs:SameType}) together yield \ref{hyp:SameType}.

\medskip
We now prove that \ref{hyp:ContinuityFSystems} holds.
We apply the strategy explained in Subsection \ref{sec:VerifHyp} with the following two lemmas:

\begin{lemma}
Let $(p_w(x_w), \opart{x_w})_{ w\in W}$ be an algebraically closed strongly fair amalgamation system over a $P$-independent set $Q$.
Then $(\opart{p_w}(\opart{x_w}), w\in W)$ is an amalgamation system over $Q$.
\end{lemma}

\begin{proof}
Note that axiom \ref{axiom:AmalgSyst1} and axiom  \ref{axiom:AmalgSystIndependence} (independence) of amalgamation systems are direct consequences of their analog for the amalgamation system $(p_w, w\in W)$.
Axiom \ref{axiom:AmalgSyst2} comes from the compatibility of the family $(\opart{x_w}, w\in W)$.
It only remains to show that axiom  \ref{axiom:AmalgSystControl} (controlled character) hold.

Let then $b_w$ be a realization of $p_w$ witnessing the strong fairness.
We therefore have that $\bigcup_{i\in w} b_i$ is $P$-independent, hence the assumption (\ref{pairs:AlgClosure}) implies that $\acl_1(b_i, i\in w)$ equals $\acl_0(b_i, i\in w)$.
In particular, the tuple $b_w$ is contained in $\acl_0(b_i,i\in w)$.
By $P$-independence of $\bigcup_{i\in w} b_i$ and since $\opart{b_i} = P_{b_i}$, it follows that $b_w\ind^0_{(\opart{b_i}, i\in w)} P_M$ holds.
This yields that $\opart{b_w}=P_{b_w}$ is algebraic over $(\opart{b_i}, i\in w)$.
\end{proof}

Keeping the notations of the lemma, if $p_\top(x_\top)$ is a minimal completion of the algebraically closed strong fair amalgamation system $(p_w(x_w), \opart{x_w})_{ w\in W}$, then $\opart{p_\top}$ is a minimal completion of $(\opart{p_w}, w\in W)$.
Recall that the $\opart{p_w}$'s do not fork over $P_Q$.
We showed in Remark \ref{rk:ContinuityNonforking} that the property "being nonforking over a given subset" is preserved by amalgamation, whence $\opart{p_\top}$ does not fork over $P_Q$.
Moroever, this also implies by Lemma \ref{lem:RestrictionAmalgSystem} that $(\widehat{p_w}, w\in W)$ is an amalgamation system over $P_Q$, and the fact that $\opart{p_\top}$ does not fork over $P_Q$ yields that $\widehat{p_\top}$ is a completion of $(\widehat{p_w}, w\in W)$.

Note also that by fairness of the system, every $p_w(x_w)$ has a realization $a_w$ such that $\opart{a_w} = P_{a_w}$ is relatively algebraically closed in $P$.
In particular, the tuple $\opart{a_w}$ is a relatively algebraically closed realization of $\widehat{p_w}$; it follows that $(\widehat{p_w}, w\in W)$ is an amalgamation system which is algebraically closed relatively to $P$.

In the case where $\mcC(P_Q)$ is closed under relatively algebraically closed amalgamation, this implies that $\widehat{p_\top}$ can be realized in $P$.
We can now use the assumption (\ref{pairs:PIndependentRealizations}) to find a realization $b_\top$ of $p_\top$ which is independent from $P$ over $\opart{b_\top}$, with $\opart{b_\top}$ contained in $P$ (i.e. $(p_\top, \opart{x_\top})$ is fair).

\begin{lemma} \label{lem:DissociationPOSystems}
For every strict subsets $v_1,...,v_m$ of $[n]$, we have the independence $b_{v_1}...b_{v_m} \ind^0_{\opart{b_{v_1}}...\opart{b_{v_1}}} P_M$.
\end{lemma}

\begin{proof}
Define $u=[n]\setminus \bigcup_i v_i$.
By choice of $b_\top$ we have $b_\top\ind^0_{\opart{b_\top}} P_M$, 
and the previous lemma yields $\opart{b_\top} \sseq \acl_0(\opart{b_{v_1}}...\opart{b_{v_m}} \opart{b_u})$, so we can deduce that the independence $b_{v_1}...b_{v_m} \ind^0_{\opart{b_{v_1}}...\opart{b_{v_m}} \opart{b_u}} P_M$ holds.
Since $p_\top$ is a completion of an independent amalgamation system, we know that  $b_{v_1}...b_{v_m} \ind^0_Q b_u$ holds, so  $b_{v_1}...b_{v_m} \ind^0_{Q\opart{b_{v_1}}...\opart{b_{v_m}}} P_M$ follows by transitivity.
Finally, we know that $Q$ is $P$-independent and that $\opart{b_{v_1}}...\opart{b_{v_m}}$ is contained in $P_M$, so transitivity yields the result.
\end{proof}

For every $w$ in $W$, we get that $\opart{b_w}$ is a realization of $\widehat{p_w}$ contained in $P$, which yields that it is relatively algebraically closed by the assumption (\ref{pairs:RelAlgClos}).
Moreover, the lemma implies the independence $b_w \ind^0_{\opart{b_w}} P_M$, therefore $\opart{b_w}$ coincide with $P_{b_w}$.
Lemma \ref{lem:DissociationPOSystems} hence states that $(p_w, w\in W) \cup \{ p_\top \}$ is a strong fair amalgamation system.
We can therefore apply Theorem \ref{thm:AbstractAmalgamationTransfer} to obtain the following:

\begin{cor} \label{cor:nExistPairs}
Let $Q$ be an algebraically closed $P$-independent set.
If  $T_0$ has $n$-existence over $Q$ and if $\mcC(P_Q)$ is closed under relatively algebraically closed $n$-amalga-mation, then $T_1$ has $n$-existence over $Q$.
\hfill \qedsymbol
\end{cor}

Let us now provide examples of theories to which Corollary  \ref{cor:nExistPairs} applies.
We first consider the theory of \emph{lovely pairs} of a simple theory $T_0$, introduced in \cite{Poizat_1983} under the name "belles paires" in the case of $T_0$ stable and generalized to $T_0$ simple in \cite{BenyaacovPillayVassiliev2003}.

Recall that for $M$ a model of a simple theory $T_0$ and $P$ a new unary predicate such that $P_M$ is a model of $T_0$, we say that the pair $(M, P_M)$ is a $\kappa$\textit{-lovely pair} if
\begin{enumerate}
	\item[($i$)] \label{AxiomLovelyPairs1} For all $C$ in $M$ with $|C|<\kappa$ and all finitary $\mcL$-type $p$ over $C$, there exists a realization $a$ of $p$ in $M$ such that $a\ind_C^0 P_M$ holds.
	\item[($ii$)] \label{AxiomLovelyPairs2} For all $C$ in $M$ with $|C|<\kappa$ and all finitary $\mcL$-type $p$ over $C$, if $p$ does not fork over $P_C$, then $p$ is realized in $P_M$.
\end{enumerate}
\noindent
If $\kappa= |T_0|^+$, we just say that $(M, P_M)$ is a \textit{lovely pair}.

We say that the theory of lovely pairs of $T_0$ exists if every $\kappa$-saturated model of $T_1$ is a $\kappa$-lovely pair of models of $T_0$, where $T_1$ is the common $\mcL_P$-theory of all lovely pairs of models of $T$.

From \cite[Remark 7.2 and Proposition 7.3]{BenyaacovPillayVassiliev2003}, one obtains the following results: 

\begin{fact} 
\label{fact:SummaryBP}
Let $A, B$ and $C$ be $P$-independent subsets of $M$, with $C\sseq A\cap B$. The following hold:
\begin{enumerate}[(a)]
	\item Any set which is algebraically closed in $T_1$ is $P$-independent.
	
	\item The algebraic closures in $T_0$ and in $T_1$ coincide: $\acl_1(A)= \acl(A)$.
	
	\item Independence is characterized as follows
	$$ A\ind_C^1 B \Leftrightarrow  A\ind_{P,C}^0 B \text{ and } A \ind_C^0 B$$

	\item If $A$ and $B$ have the same quantifier-free $\mcL_1$-type over $C$ with $P_A$ and $P_B$ algebraically closed in $T_0$, then $A$ and $B$ have the same $\mcL_1$-type over $C$.
	
\end{enumerate}

\end{fact}

Note that the assumptions (\ref{pairs:AlgClosure}), (\ref{pairs:RelAlgClos}) and (\ref{pairs:SameType}) of this subsection are clear consequences of Fact \ref{fact:SummaryBP}.
Moreover, simple forking calculus allows to deduce (\ref{pairs:CaracIndep}) from the caracterisation of the independence relation.
Finally, if $p(x)$ is a type in $T_0$ over a $P$-independent set and $\opart x$ a subtuple of $x$ such that $p(x)\rest \opart x$ does not fork over $P_Q$, then axiom ($ii$) of lovely pairs yields that $p(x)\rest \opart x$ has a realization $\opart a$ in $P_M$.
Now axiom ($i$) yields a realization $a$ of the type $p(x) \cup\{\opart{x} = \opart{a}\}$ such that $a\ind_{Q\opart{a}}^0 P_M$ holds.
It follows that $a$ is independent from $P_M$ over $P_Q \opart a$ hence (\ref{pairs:PIndependentRealizations}) holds.

If $D$ is an algebraically closed subset of $P_M$, the class $\mcC(D)$ is simply the class of all $\mcL_0$-types over $D$ - in particular it is closed under amalgamation. 
We can therefore apply Corollary \ref{cor:nExistPairs}, which yields: 

\begin{cor} \label{cor:nAmalgLovelyPairs}
Let $T_0$ be a simple theory with QE and EI such that the theory $T_1$ of lovely pairs of models of $T$ exists. 
Given $Q$ a subset of a model of $T_1$ closed under $\acl_P$, if $T_0$ has $n$-existence over $Q$, then $T_1$ has $n$-existence over $Q$.
\hfill \qedsymbol
\end{cor}

\smallskip

Assume now that $T_0$ is a supersimple $\mcL_0$-theory and let $\mcF$ be a type-definable set over $\emptyset$ in $T_0$.
In \cite{BerensteinCarmonaVassiliev_2017}, the authors define the theory $T_1= (T_0)_\mcF^{ind}$ of $H$-structures associated to $\mcF$, provided that $T_0$ eliminates $\exists^{large}$ and that the extension property is first-order (see \cite{BerensteinCarmonaVassiliev_2017}).
We keep the notations of \cite{BerensteinCarmonaVassiliev_2017}, and we call $H$ the new predicate symbol instead of $P$.
A $H$\textit{-structure associated to} $\mcF$ is a structure $(M, H_M)$ such that 
\begin{enumerate}[(i)]
	\item[($i$)] $M$ is a model of $T_0$;
	\item[($ii$)] $H$ forms an independent set of realizations of $\mcF$;
	\item[($iii$)] If $A\sseq M$ is finite and $q\in S_1(A)$ is a nonforking extension of $\mcF$ to $A$, then there are realizations $b$ and $c$ of $q$ such that $b$ belongs to $H_M$ and $c$ is independent from $H_M$ over $A$.
\end{enumerate}
All $H$-structures associated to $\mcF$ are elementarily equivalent and $T_1$ is their common theory.
All $|T|^+$-saturated models of $T_1$ are $H$-structures; let us fix such a structure $(M, H_M)$.
We keep the notation of Definition \ref{def:PIndependentSet} (adapted to the case where the predicate $P$ is called $H$) and  we say that a tuple $a$ is $H$-independent if $a\ind_{H_a}^0 H_M$ holds.

\begin{fact} \label{fact:SummaryHStructures}
We gather here some results about $H$-structures.
\begin{enumerate}[(a)]
	\item \cite[Lemma 2.6]{BerensteinCarmonaVassiliev_2017} Any model $M$ of $T_0$ with a distinguished independent subset $H(M)$ of realizations of $\mcF$ can be embedded into a $H$-structure in a $H$-independent way.
	
	\item \cite[Proposition 2.7]{BerensteinCarmonaVassiliev_2017} Let $a$ and $a'$ be $H$-independent tuples.
	If $\qftp_1(a)=\qftp_1(a')$, then $\tp_1(a)=\tp_1(a')$.
	
	\item \cite[Lemma 4.12]{BerensteinCarmonaVassiliev_2017} Let $A$ be $H$-independent. 
	The algebraic closures of $A$ in $T_0$ and in $T_1$ coincide: $\acl_0(A)=\acl_1(A)$.
	
	\item \cite[Theorem 5.3]{BerensteinCarmonaVassiliev_2017} Let $a$ be a tuple in $M$ and $C$ a $H$-independent subset of $M$.
	Then there exists a smallest subset $H_0$ of $H$ such that $a\ind_{C H_0}^0 H$; we call it the $H$-basis of $a$ over $C$ and we denote it $\hb(a/C)$.
	Moreover, for all tuple $a$ of $M$, we have $\acl_1(a)=\acl_0(a, \hb(a))$. 
	For all $\acl_1$-closed subsets $C\sseq B$, we have
	\[
		a\ind^1_C B \Leftrightarrow a\ind_{H_M C}^0 B \text{ and } \hb(a/C)=\hb(a/B).
	\]	
\end{enumerate}
\end{fact}

Note that all subsets of $H_M$ are relatively algebraically closed in $H_M$; the assumptions (\ref{pairs:AlgClosure}), (\ref{pairs:RelAlgClos}) (\ref{pairs:SameType}) and (\ref{pairs:PIndependentRealizations}) are  therefore clear consequences of Fact \ref{fact:SummaryHStructures}.
Moreover, simple forking-calculus coupled with the definition of the $H$-basis  allows to deduce (\ref{pairs:CaracIndep}) from the caracterization of $\ind^1$.

Remark that for all $D$ in $H_M$, the class $\mcC(D)$ of types over $D$ with a realization in $H$ is closed under  amalgamation.
To see this, let $S=(p_w(x_w), w\in W)$ be an amalgamation system over $D$ such that for all $w$, the type $p_w$ is realizable in $H$.
Take a realization $a_w$ of $p_w$ in $H$, for some $w\in W$.
Then $a_w$ is algebraic over the $a_i$'s, $i\in w$.
Since the relative algebraic closure in $H$ is trivial, this implies that $a_w=\bigcup_{i\in W} a_i$.
It follows that a minimal completion of $S$ is simply the type of an independent family of realizations of the $p_i$'s; in particular such a type is realizable in $H$.

This proves that for all $T_1$-algebraically closed subset $Q$ of $M$, if $T_0$ has $n$-existence over $Q$, then $T_1$ does so too.
We thus obtain:

\begin{cor}
Let $T_0$ be a supersimple theory with QE and EI such that $T_0$ eliminates $\exists^{large}$ and  the extension property is first-order. 
Let $\mcF$ be a type-definable set over $\emptyset$ in $T_0$.
Let $Q$ be an algebraically closed subset of a model of the theory $T_1$ of $H$-structures associated to $\mcF$.
If $T_0$ has $n$-existence over $Q$, then so does $T_1$.
\hfill $\square$
\end{cor}

\bigskip

Finally, we consider the case of a $\kappa$-PAC beautiful pairs of a stable theory $T_0$, defined by Polkowska in in \cite{Polkowska_2007}.
Let $T_0$ be a stable theory with quantifier elimination, nfcp, PACFO and TACSS.
Fix a $\kappa$-PAC beautiful pair, i.e. a pair $(M, P_M)$ where $P$ is a $\kappa$-PAC substructure of the model $M$ of $T_0$ and such that $M$ is saturated over $P_M$.
By PACFO we may assume $(M, P_M)$ to be sufficiently saturated saturated.
Take $T_1$ to be the $\mcL_1$ theory of $(M,P_M)$ over an elementary substructure. 

\begin{fact} \label{fact:SummaryPACBP}
Let $A, B$ and $C$ be $P$-independent subsets of $M$, with $C\sseq A\cap B$. The following hold:
\begin{enumerate}[(a)]
	\item Any set which is algebraically closed in $T_1$ is $P$-independent.
	
	\item The algebraic closures in $T_0$ and in $T_1$ coincide: $\acl_1(A)= \acl_0(A)$.
	
	\item Independence is characterized as follows
	$$ A\ind_C^1 B \Leftrightarrow  A\ind_{P,C}^0 B \text{ and } A \ind_C^0 B$$

	\item If $A$ and $B$ have the same quantifier-free $\mcL_P$-type over $C$ with $P_A$ and $P_B$ algebraically closed in $T$, then $A$ and $B$ have the same $\mcL_P$-type over $C$.
\end{enumerate}
\end{fact}

Again, Fact \ref{fact:SummaryPACBP} immediately yields (\ref{pairs:AlgClosure}), (\ref{pairs:RelAlgClos}) and (\ref{pairs:SameType}).
Moreover, we obtain (\ref{pairs:CaracIndep}) and (\ref{pairs:PIndependentRealizations}) in a similar way to the case of lovely pairs.

Finally, we use Lemma \ref{prop:RelNUniq} to see that if $D$ is an existentially closed subset of $P_M$, then any minimal completion to a relatively algebraically closed amalgamation system is stationary and hence realizable in $P$.
Thus, we have proved:

\begin{cor}
Let $T_0$ be a stable nfcp theory with quantifier elimination and PACFO.
Consider a theory $T_1$ of a bounded PAC beautiful pair in $T_0$ such that $(\star)$ holds.
If $Q$ is an existentially closed subset of a model of $T_1$ containing an elementary substructure of $P$ such that $T_0$ has $n$-existence over $Q$, then $T_1$ has $n$-existence over $Q$.
\hfill \qedsymbol
\end{cor}

\subsection{Generic predicate} 			\label{sec:GenericPredicate}
We conclude by providing another example of theory $T_1$ obtained by adding a new unary predicate to the theory $T_0$ to which Theorem \ref{thm:AbstractAmalgamationTransfer} applies.
We did not include it in the previous subsection, as it does not meet the hypotheses of Corollary \ref{cor:nExistPairs}.
Let $T_0$ be a complete simple $\mcL_0$-theory with quantifier elimination and elimination of $\exists^\infty$, and define $\mcL_1=\mcL_0\cup\{P\}$, where $P$ is a new unary predicate.
Chatzidakis and Pillay showed in \cite{Chatzidakis_Pillay_1998} that under these conditions, the theory $T_0$, considered in the language $\mcL_1$, has a model companion $T_{0,P}$.
Moreover, the completions of $T_{0,P}$ are given by describing $P\cap \acl_0(\emptyset)$. 
Let $T_1$ be one such completion and fix a $\kappa$-saturated model $(M, P_M)$ of $T_1$, with $\kappa > |T|$.
We assume without loss of generality that both $T_0$ and $T_1$ contain $P_M\cap \acl_0(\emptyset)$ as parameters.

\begin{fact} \cite[Corollaries 2.6]{Chatzidakis_Pillay_1998}
\label{fact:SummaryGenericPredicate}
\begin{enumerate}[(a)]
	\item The algebraic closure in $T_1$ coincides with the algebraic closure in $T_0$.
	\item Any two algebraically closed subsets of $M$ have the same $\mcL_1$-type if and only if they have the same $\mcL_1$-quantifier free type.
\end{enumerate}
\end{fact}

In \cite{Chatzidakis_Pillay_1998}, the authors showed that $T_1$ is simple and that $\ind^1$ coincide with $\ind^0$.
Theorem \ref{thm:AbstractAmalgamationTransfer} allows us to describe an easy generalization of their proof of the independence theorem in $T_1$, showing that $\ind^1$ has $n$-existence for all integer $n$.

\medskip
Since any model of $T_1$ is a model of $T_0$, the hypotheses \ref{hyp:StrongBdd} and \ref{hyp:CoheirLemmaLight} are immediate.
Define $\mcG$ to simply be the class of all small subsets of $M$.
The hypotheses \ref{hyp:InteractionAlgClosure} - \ref{hyp:SameType} are then immediate consequences of Fact \ref{fact:SummaryGenericPredicate} an of the characterization of $\ind^1$.
Moreover, we can see by genericity of the predicate that all the annotated types $(p(x), \tilde{x})$ are fair, and \ref{hyp:ContinuityFSystems} follows immediately.

\begin{cor}
Let $T_0$ be a simple theory with quantifier elimination and elimination of $\exists^\infty$. 
Denote by $T_1$ the theory of any model of $T_0$ endowed with a generic predicate.
The theory $T_1$ is simple and has $n$-existence over an algebraically closed subset $A$ whenever $T_0$ has $n$-existence over $A$.
\hfill \qedsymbol
\end{cor}


\smallskip

\begin{thebibliography}{10}

\bibitem{Ax1973}
{\sc Ax, J.}
\newblock The elementary theory of finite fields.
\newblock {\em Journal of Symbolic Logic 38}, 1 (1973), 162--163.

\bibitem{bartnick_2024}
{\sc Bartnick, C.}
\newblock Stationarity and elimination of imaginaries in stable and simple
  theories.
\newblock {\em Fundamenta Mathematicae 270\/} (2025), 277--299.

\bibitem{BenyaacovPillayVassiliev2003}
{\sc Ben-Yaacov, I., Pillay, A., and Vassiliev, E.}
\newblock Lovely pairs of models.
\newblock {\em Annals of Pure and Applied Logic 122}, 1 (2003), 235--261.

\bibitem{BerensteinCarmonaVassiliev_2017}
{\sc Berenstein, A., Carmona, J.F.,  and Vassiliev, E.}
\newblock Supersimple structures with a dense independent subset.
\newblock {\em Mathematical Logic Quarterly 63}, 6 (2017), 552--573.


\bibitem{BlossierHardoinMP_2017}
{\sc Blossier, T., Hardouin, C., and Martin-Pizarro, A.}
\newblock Sur les automorphismes born{\'e}s de corps munis d’op{\'e}rateurs.
\newblock {\em Mathematical Research Letters 24}, 4 (2017), 955--978.

\bibitem{Blossier_MartinPizarro_2019}
{\sc Blossier, T., and Martin-Pizarro, A.}
\newblock {Un critère simple}.
\newblock {\em Notre Dame Journal of Formal Logic 60}, 4 (2019), 639 -- 663.





\bibitem{Chatzidakis_1999}
{\sc Chatzidakis, Z.}
\newblock { Simplicity and Independence for Pseudo-Algebraically Closed
  Fields}.
\newblock {\em London Mathematical Society Lecture Note Series.} Cambridge University Press (1999) p.~41–62.
  

\bibitem{Chatzidakis_2001}
{\sc Chatzidakis, Z.}
\newblock Generic automorphisms of separably closed fields.
\newblock {\em Illinois Journal of Mathematics 45}, 3 (2001), 693--733.



\bibitem{ChatzidakisHrushovski_1999}
{\sc Chatzidakis, Z., and Hrushovski, E.}
\newblock Model theory of difference fields.
\newblock {\em Transactions of the American Mathematical Society 351}, 8 (1999), 2997--2071.
  

  

\bibitem{Chatzidakis_2004}
{\sc Chatzidakis, Z., and Hrushovski, E.}
\newblock Model theory of endomorphisms of separably closed fields.
\newblock {\em Journal of Algebra 281}, 2 (2004), 567--603.

\bibitem{Chatzidakis_Pillay_1998}
{\sc Chatzidakis, Z., and Pillay, A.}
\newblock Generic structures and simple theories.
\newblock {\em Annals of Pure and Applied Logic 95}, 1 (1998), 71--92.

\bibitem{DePiro_Kim_Millar_2005}
{\sc de~Piro, T., Kim, B., and Millar, J.}
\newblock Constructing the hyperdefinable group from the group configuration.
\newblock {\em Journal of Mathematical Logic 6\/} (2005).

\bibitem{Delon_1988}
{\sc Delon, F.}
\newblock Idéaux et types sur les corps séparablement clos.
\newblock {\em Mémoires de la Société Mathématique de France. Nouvelle
  Série 33\/} (1988).

\bibitem{Ershov_1967}
{\sc Ershov, Y.}
\newblock Fields with a solvable theory.
\newblock {\em Soviet Mathematics. Doklady 8\/} (1967).

\bibitem{Hoffmann_Kowalski_2023}
{\sc Hoffmann, D., and Kowalski, P.}
\newblock {PAC} structures as invariants of finite group actions.
\newblock {\em The Journal of Symbolic Logic\/} (2023), 1--34.

\bibitem{hrushovski_2012}
{\sc Hrushovski, E.}
\newblock Groupoids, imaginaries and internal covers.
\newblock {\em Turkish Journal of Mathematics 36, 2} (2012),  173--198.


\bibitem{HrushovskiPFFieldsRelStructures}
{\sc Hrushovski, E.}
\newblock Pseudo-finite fields and related structures.
\newblock {\em Model theory and applications}, vol.~11 of {\em Quad. Mat.}
  Aracne, Rome (2002) pp.~151--212.
  

\bibitem{KimPillay1997}
{\sc Kim, B., and Pillay, A.}
\newblock Simple theories.
\newblock {\em Ann. Pure Appl. Log. 88\/} (1997), 149--164.

\bibitem{Kolesnikov_2004_thesis}
{\sc Kolesnikov, A.}
\newblock { Generalized amalgamation in simple theories and characterization
  of dependence in non-elementary classes}.
\newblock {\em Carnegie Mellon University} (2004).

\bibitem{lang2012algebra}
{\sc Lang, S.}
\newblock {\em Algebra}, vol.~211.
\newblock Springer Science \& Business Media (2012).

\bibitem{Sanchez_Tressl_2020_DiffLargeFields}
{\sc Le{\'o}n~S{\'a}nchez, O., and Tressl, M.}
\newblock Differentially large fields.
\newblock {\em Algebra \& Number Theory 18}, 2 (2024), 249--280.

\bibitem{Ludwig_2026}
{\sc Ludwig, S.~M.}
\newblock Higher amalgamation in $\mathrm{ACFA}^{+}$.
\newblock {\em arXiv:2601.05096\/} (2026).

\bibitem{marker}
{\sc Marker, D.}
\newblock {\em Model theory: an introduction}, vol.~217.
\newblock Springer Science \& Business Media (2002).

\bibitem{Marker2000ModelTO}
{\sc Marker, D., et~al.}
\newblock Model theory of differential fields.
\newblock {\em Model theory of fields 5\/} (1996), 38--113.

\bibitem{MartinPizarro_Bourbaki}
{\sc Martin-Pizarro, A.}
\newblock Model theory, differential algebra and functional transcendence,
  after {F}reitag, {J}aoui, and {M}oosa.
\newblock {\em Séminaire Bourbaki n°1245\/} (2025).

\bibitem{Mcgrail_2000}
{\sc McGrail, T.}
\newblock The model theory of differential fields with finitely many commuting
  derivations.
\newblock {\em The Journal of Symbolic Logic 65}, 2 (2000), 885--913.

\bibitem{Medina_2007}
{\sc Medina, R.~B.}
\newblock Differentially closed fields of characteristic zero with a generic
  automorphism.
\newblock {\em Revista de Matem{\'a}tica: Teor{\'\i}a y Aplicaciones 14}, 1
  (2007), 81--100.

\bibitem{MoosaScanlon_2014}
{\sc Moosa, R., and Scanlon, T.}
\newblock Model theory of fields with free operators in characteristic zero.
\newblock {\em Journal of Mathematical Logic 14}, 02 (2014), 1450009.


\bibitem{Pillay_Polkowska_2006}
{\sc Pillay, A., and Polkowska, D.}
\newblock On {PAC} and bounded substructures of a stable structure.
\newblock {\em Journal of Symbolic Logic 71\/} (2006).

\bibitem{Poizat_1983}
{\sc Poizat, B.}
\newblock Paires de structures stables.
\newblock {\em Journal of Symbolic Logic 48}, 2 (1983), 239–249.

\bibitem{Polkowska_2007}
{\sc Polkowska, O. P. N.~M.}
\newblock On simplicity of bounded pseudoalgebraically closed structures.
\newblock {\em Journal of Mathematical Logic 7}, 2 (2007), 173--193.

\bibitem{Shelah_83_ClassificationTheory}
{\sc Shelah, S.}
\newblock Classification theory for nonelementary classes {I}: {T}he number of
  uncountable models of $\psi \in {L}_{ \omega_1, \omega}$. {P}art {B}.
\newblock {\em Israel Journal of Mathematics 46\/} (1983), 241--273.

\bibitem{Srour_IndependenceSCF}
{\sc Srour, G.}
\newblock The independence relation in separably closed fields.
\newblock {\em The Journal of Symbolic Logic 51}, 3 (1986), 715--725.

\bibitem{Wood_1976}
{\sc Wood, C.}
\newblock The model theory of differential fields revisited.
\newblock {\em Israel Journal of Mathematics 25}, 3 (1976), 331--352.

\end{thebibliography}
\end{document}